\documentclass[12pt]{article}
\usepackage[margin=1in]{geometry}
\usepackage{tocloft}
\usepackage[font=small]{caption}
\usepackage{amsmath,amssymb,mathrsfs,tensor,bbold,upgreek,cases}
\usepackage{amsthm,amsfonts,mathtools,bm,slashed}
\usepackage{fontawesome}
\usepackage{listings,algorithm,algpseudocode}
\usepackage{afterpage,placeins}
\usepackage{titlefoot}
\usepackage{float,graphicx,comment,subfig}
\graphicspath{{./}{figs/}}
\usepackage[dvipsnames]{xcolor}
\colorlet{MyBlue}{black}
\usepackage[linktocpage=true,colorlinks=true,
            citecolor=MyBlue,linkcolor=MyBlue,urlcolor=MyBlue]{hyperref}
\usepackage{cite}
\newtheorem{theorem}{Theorem}[section]
\newtheorem{definition}[theorem]{Definition}

\newtheorem{corollary}[theorem]{Corollary}

\theoremstyle{remark}

\numberwithin{equation}{section} 

\newlength\hrulethickness
\makeatletter
\renewcommand\fs@ruled{\def\@fs@cfont{\bfseries}\let\@fs@capt\floatc@ruled
  \def\@fs@pre{\hrule height \hrulethickness depth0pt \kern2pt}%
  \def\@fs@post{\kern2pt\hrule height \hrulethickness depth0pt \relax}%
  \def\@fs@mid{\kern2pt\hrule height \hrulethickness depth0pt \kern2pt}%
  \let\@fs@iftopcapt\iftrue}
\makeatother

\newcommand\blfootnote[1]{%
  \begingroup
  \renewcommand\thefootnote{}\footnote{#1}%
  \addtocounter{footnote}{-1}%
  \endgroup
}

\newcommand{\Reals}{\mathbb{R}}

\newcommand\mailto[1]{\href{mailto:#1}{\tt #1}}

\DeclareMathOperator{\diag}{diag}

\DeclareMathOperator{\vol}{vol}

\begin{document}

\thispagestyle{empty}
\renewcommand{\thefootnote}{\fnsymbol{footnote}}

\vspace*{2cm}
\begin{center}

{\LARGE \bf Conformal Symplectic and Relativistic\\[.4em]Optimization}

\vspace{1cm}

{\bf Guilherme Fran\c ca,$\!$%
\footnote{\faEnvelopeO~\mailto{guifranca@gmail.com}}$^{,1,2}$%
\blfootnote{\faCalendar~December 2020}
Jeremias Sulam,$^{\!2}$~Daniel P. Robinson,$^{\!3}$ and Ren\' e Vidal$^{\,2}$
}

\vspace{.8cm}

{\it
${}^{1}$Division of Computer Science, 
University of California, Berkeley, CA, USA\\[.2em]
${}^{2}$Mathematical Institute for Data Science, 
Johns Hopkins University, MD, USA\\[.2em]
${}^{3}$Industrial and Systems Engineering, 
Lehigh University, PA, USA
}

\vspace{1.5cm}

{\bf Abstract}
\vspace{-1em}

\end{center}
\noindent
Arguably, the two most popular accelerated or momentum-based optimization
methods in machine learning are Nesterov's accelerated gradient and Polyaks's heavy ball,
both corresponding to different discretizations of a particular second order
differential equation with friction. Such connections with
continuous-time dynamical systems have been instrumental in demystifying
acceleration phenomena in optimization.
Here we  study structure-preserving discretizations for a certain class of
dissipative (conformal) Hamiltonian systems, allowing us to analyze the
symplectic structure of both Nesterov and heavy ball, besides providing
several new insights into these methods.
Moreover,  we propose a new algorithm based on a dissipative relativistic
system that normalizes the momentum and may result in more stable/faster
optimization. Importantly, such a method generalizes both Nesterov and heavy
ball, each being recovered as distinct limiting cases, and has potential
advantages at no additional cost.

\renewcommand*{\thefootnote}{\arabic{footnote}}
\setcounter{footnote}{0}
\setcounter{page}{0}

\newpage
\tableofcontents

\bigskip

\section{Introduction}

Gradient based optimization methods are ubiquitous
in machine learning since they
only require first order information on the objective function. This
makes them
computationally efficient. However, vanilla gradient descent can be slow.
Alternatively,
\emph{accelerated gradient methods}, whose  construction can be traced
back to
Polyak \cite{Polyak:1964} and Nesterov \cite{Nesterov:1983}, became popular
due to their ability to achieve best worst-case complexity bounds.
The heavy ball method, also known as
\emph{classical momentum} (CM) method, is given by
\begin{equation}
\label{momgd2}
v_{k+1} = \mu v_k - \epsilon \nabla f(x_k),  \qquad
x_{k+1} = x_k + v_{k+1},
\end{equation}
where $k=0,1,\dotsc$ is the iteration number, $\mu \in (0,1)$ is the momentum factor, $\epsilon > 0$ is the
 learning rate, and $f:\mathbb{R}^n \to \mathbb{R}$ is the
function being minimized. Similarly, \emph{Nesterov's accelerated gradient} (NAG)
can be found in the form
\begin{equation}
\label{nag2}
v_{k+1} = \mu v_k - \epsilon \nabla f(x_k + \mu v_k), \qquad
x_{k+1} = x_k + v_{k+1}.
\end{equation}
Both methods have a long history in optimization
and machine learning \cite{Hinton:2013}. They are also
the basis for the construction of other methods, such as adaptive ones that additionally
include some gradient normalization \cite{Tieleman:2012,Kingma:2015,Duchi:2017,Dozat:2016}.

In discrete-time optimization the ``acceleration phenomena'' are considered
counterintuitive.
By this we mean a mechanism by which an algorithm can be accelerated, i.e. have a faster convergence;
for instance, it is known that gradient descent converges at a rate of $O(1/k)$ for convex functions, while
NAG converges at a rate $O(1/k^2)$, which is optimal in the sense of worst-case complexity.
A complete understanding of why NAG is able to achieve such an improved rate is considered by
many experts an important open problem, and currently there is no guiding principle to construct accelerated algorithms.
A promising direction to understand this has been emerging
in connection with continuous-time dynamical systems
\cite{Candes:2016,Wibisono:2016,Krichene:2015,
Zhang:2018,Shi:2018,Yang:2018,Betancourt:2018,
Franca:2018,Franca:2018b,Franca:2019_split,Franca:2020} where many of these difficulties disappear or have an intuitive explanation.
Since one is free to discretize a continuous-time system in many different ways, it is
only natural to ask which discretization strategies would be most suitable for optimization?
Such a question is unlikely to have a simple answer, and may be problem dependent.
Unfortunately, typical discretizations are also known to introduce spurious artifacts
and do not reproduce the most important properties
of the continuous-time system \cite{McLachlan:2006}.
Nevertheless, a special class of discretizations in the physics literature known
as \emph{symplectic integrators} \cite{Forest:2006,Channell:1990,McLachlan:2006,Quispel:2006}
are to be preferable whenever considering the special class of \emph{conservative Hamiltonian systems}.

More relevant to optimization is a class of \emph{dissipative} systems
known as  \emph{conformal Hamiltonian systems} \cite{McLachlan:2001}.
Recently, results from symplectic integrators were
 extended to this case and such methods are called
\emph{conformal symplectic integrators}  \cite{Moore:2016,Franca:2020}.
Conformal symplectic methods tend to have long time stability because the numerical trajectories remain
in the same conformal symplectic manifold as the original  system \cite{Franca:2020}.
Importantly, these methods do not
change the phase portrait of the system, i.e. the stability of critical points is preserved.
Although symplectic techniques have had great success in several areas of physics
and Monte Carlo methods, only  recently they started to be considered in
optimization \cite{Betancourt:2018,Franca:2020} and are still mostly unexplored in this context.
Very recently a great progress has been made  \cite{Franca:2020} by showing that such an approach is able to preserve the continuous-time rates of convergence up to a controlled error \cite{Franca:2020}.

In this paper, we \emph{relate conformal symplectic integrators to optimization} and provide important insights
into CM \eqref{momgd2} and NAG \eqref{nag2}. 
We prove that CM is a first order accurate conformal symplectic  integrator.
On the other hand, we show that NAG is also first order accurate,  but  not conformal symplectic since it introduces some spurious dissipation---or excitation.
However, it does so in an interesting way that
depends on the Hessian $\nabla^2 f$;  the symplectic form contracts
in a Hessian dependent manner and so do phase space volumes.  This is an effect of higher order but can influence the
behaviour of the algorithm.
We also derive \emph{modified equations} and \emph{shadow Hamiltonians} for both CM and NAG.
Moreover, we indicate a tradeoff between
stability, symplecticness, and such an spurious contraction, indicating
advantages in structure-preserving discretizations for optimization.

Optimization can be challenging in a landscape with large gradients, e.g. for a function with fast growing tails. The only way to control divergences  in
methods such as \eqref{momgd2} and \eqref{nag2} is to make the step size very
small, but then the algorithm becomes slow. One approach to this issue is to introduce a suitable normalization of the gradient.
Here we propose an alternative approach motivated by \emph{special relativity} in physics.
The reason is that in special relativity there is a limiting speed, i.e.
the speed of light.   Thus, by discretizing  a \emph{dissipative relativistic system},
we obtain an algorithm that incorporates this effect and may result in more stable optimization in
settings with large gradients. Specifically, we introduce Algorithm~\ref{rgd1}.
Besides the momentum factor $\mu$ and the learning rate $\epsilon$---also present
in \eqref{momgd2} and \eqref{nag2}---the above \emph{relativistic gradient descent} (RGD) method has the additional parameters
$\delta \ge 0$ and $0 \le \alpha \le 1$ which brings some interesting properties:

\begin{algorithm}[t]
\caption{RGD method for minimizing a smooth function $f(x)$. \label{rgd1}
In practice, we recommend setting $\alpha=1$ which results in a conformal symplectic
method.
}
\begin{algorithmic}
\Require Initial state $(x_0, v_0)$ and parameters
$\epsilon > 0$, $\delta > 0$, $\mu \in (0,1)$, $\alpha \in [0,1]$
\For{$k=0,1,\dotsc$}
\State $x_{k+1/2} \gets x_k +  \sqrt{\mu} \left( \mu \delta \| v_k\|^2 +1 \right)^{-1/2} v_k $
\State $v_{k+1/2} \gets \sqrt{\mu} v_k - \epsilon \nabla f(x_{k+1/2})$
\State $x_{k+1} \gets \alpha x_{k+1/2} + (1-\alpha) x_k + \left( \delta \| v_{k+1/2}\|^2 + 1 \right)^{-1/2}  v_{k+1/2}$
\State $v_{k+1} \gets \sqrt{\mu} \, v_{k+1/2}$
\EndFor
\end{algorithmic}
\end{algorithm}

\begin{itemize}
\item When $\delta=0$ and $\alpha=0$, RGD recovers NAG \eqref{nag2}.
When $\delta=0$ and $\alpha=1$, RGD becomes a second  order
accurate version of CM \eqref{momgd2}, which has a close behavior but an improved
stability. Thus, RGD can interpolate between these two methods.
Moreover, RGD has the same computational cost as CM or NAG. These facts imply that RGD
is at least as efficient as CM and NAG if appropriately tuned.
\item Let $y_k \equiv \alpha x_{k+1/2} + (1-\alpha) x_k$. The
last update in Algorithm~\ref{rgd1} implies $\| x_{k+1} - y_k \| \le 1/\delta$. Thus, with $\delta > 0$, RGD is globally bounded regardless how large $\|\nabla f\|$ might be; this is in contrast with CM and NAG where $\delta = 0$, i.e. $\| x_{k+1} - y_k\| \le \infty$.
The square root factor in Algorithm~\ref{rgd1} has a ``relativistic origin'' and its
strength is controlled by  $\delta$.
For this reason, RGD may be more stable than CM and NAG, preventing
divergences in settings of large gradients;
see Fig.~\ref{circles} in Section~\ref{stability_appendix} and the plots in Appendix~\ref{sec:extra_numerical}.
\item As we will show, $\alpha=1$ implies that RGD is \emph{conformal symplectic}, whereas
$\alpha =0$ implies a spurious Hessian driven damping similarly found in  NAG.
Thus, RGD has the flexibility of being ``dissipative-preserving'' or introducing
some  ``spurious contraction.''  However, based on theoretical arguments and empirical
evidence, we advocate for the choice $\alpha=1$.\footnote{The only reason for introducing the extra parameter $0 \le \alpha  \le 1$ into Algorithm~\ref{rgd1} is to actually let the experiments decide  whether $\alpha=1$ (symplectic)  or $\alpha < 1$ (non-symplectic) is desirable or not.}
\end{itemize}

Let us mention a few related works.
Applications of symplectic integrators in optimization was first considered in
\cite{Betancourt:2018}---although this is different than the conformal symplectic
case explored here.
Recently, the benefits of symplectic methods in optimization started to be indicated
\cite{Muehlebach:2020}. Actually, even more recently, a generalization of symplectic integrators
to a general class of dissipative Hamiltonian systems was proposed \cite{Franca:2020}, with theoretical results ensuring that
such discretizations are ``rate-matching'' up to a negligible
error; this construction is general and contains the conformal case considered here as a particular case.   Relativistic systems are obviously an elementary topic in physics but---with some modifications---the relativistic kinetic energy
was considered in Monte Carlo methods \cite{Lu:2017,Livingstone:2017} and also briefly
in \cite{Maddison:2018}.
Finally, we stress that Algorithm~\ref{rgd1} is a completely new method in the
literature, generalizing
perhaps the two most popular existing accelerated methods, namely  CM and NAG, and also
has the ability to be conformal symplectic besides being adaptive
in the momentum which may help controlling divergences.
We also provide several new insights into CM and NAG in Section~\ref{proof_nag_contract_omega} and Section~\ref{stability_appendix} which may be of independent interest.

\section{Conformal Hamiltonian systems}
\label{hamiltonian}

We start by introducing the basics of conformal Hamiltonian
systems and focus on their intrinsic symplectic geometry;
we refer to \cite{McLachlan:2001} for  details.
The state of the system is described by a point
on phase space, $(x, p) \in \mathbb{R}^{2n}$,
where $x=x(t)$ is
the generalized coordinates and $p=p(t)$ its
conjugate momentum, with
$t\in\mathbb{R}$ being the time.
The system is completely specified by a Hamiltonian
function $H: \Reals^{2n}\to \Reals$ and required to obey
a modified form of Hamilton's equations:
\begin{equation}
\label{confhameq}
\dot{x} = \nabla_p H(x,p), \qquad
\dot{p} = -\nabla_x H(x,p) - \gamma p.
\end{equation}
Here
$\dot{x} \equiv \tfrac{d x}{dt}$,
$\dot{p} \equiv \tfrac{d p}{dt}$,
and $\gamma > 0$ is a damping constant.
A classical example is given by
\begin{equation}
\label{hamiltonian1}
H(x,p) = \dfrac{\| p \|^2}{2m} + f(x)
\end{equation}
where $m>0$ is the mass of a particle subject
to a potential $f$.
The Hamiltonian is the energy of the system and
upon taking its time derivative
one finds $\dot{H} = - \gamma \| p \|^2 \le 0$.
Thus $H$ is a Lyapunov function and all orbits tend to critical points,
which in this case must satisfy $\nabla f(x)=0$ and $p=0$.
This implies that the system is stable on isolated minimizers of $f$.\footnote{This can be generalized for any  Hamiltonian $H$ that is strongly convex on $p$
with the minimum at $p=0$.}


Define
\begin{equation} \label{Jdef}
z \equiv \begin{bmatrix} x \\ p\end{bmatrix}, \qquad
\Omega \equiv
\begin{bmatrix} 0 & I \\ -I & 0
\end{bmatrix}, \qquad
D \equiv \begin{bmatrix} 0 & 0 \\ 0 & I \end{bmatrix},
\end{equation}
where $I$ is the $n\times n$ identity matrix,
to write the equations of motion \eqref{confhameq} concisely as\footnote{$C(z)$ and $D(z)$ will be used later on and stand
for ``conservative'' and ``dissipative'' parts, respectively.}
\begin{equation}
  \label{confhamuni}
\dot{z} = \underbrace{\Omega \nabla H(z)}_{C(z)} -\underbrace{\gamma Dz}_{D(z)} .
\end{equation}
Note that $\Omega \Omega^T = \Omega^T \Omega = I$ and $\Omega^2 = - I$, so that
$\Omega$ is real, orthogonal and antisymmetric.
Let $\xi, \eta \in \mathbb{R}^{2n}$ and define the \emph{symplectic 2-form}
$\omega(\xi,\eta)\equiv \xi^T \Omega \,\eta$.
It is convenient to use the wedge product representation of this 2-form, namely\footnote{It is not strictly necessary to be familiar with differential forms and exterior calculus to understand this paper. For the current purposes, it is enough to recall that the wedge product is a bilinear and antisymmetric operation, i.e. $dx \wedge (a dy + b dz) = a dx \wedge dy + b dx \wedge dz$
and $dx\wedge dy = -dy \wedge dx$ for scalars $a$ and $b$ and 1-forms $dx$, $dy$, $dz$ (think about this as vector differentials); we refer to \cite{Flanders} and \cite{Hairer} for more details if necessary.
}
\begin{equation}
\omega(\xi, \eta) =
(dx \wedge dp)(\xi, \eta).
\end{equation}
We  denote $\omega_t \equiv dx(t) \wedge dp(t)$.
The equations of motion define a
flow $\Phi_t: \Reals^{2n} \to \Reals^{2n}$, i.e.
$\Phi_t\big(z_0) \equiv z(t)$ where $z(0) \equiv z_0$.
Let $J_{t}(z)$ denote the Jacobian  of $\Phi_t(z)$.
From \eqref{confhamuni} it is not hard to show that
(see e.g. \cite{McLachlan:2001})%
\begin{equation}
  \label{sym_conf}
J_t^T \Omega J_t = e^{-\gamma t} \Omega
\quad \implies \quad
\omega_t = e^{-\gamma t} \omega_0.
\end{equation}
Therefore, a conformal Hamiltonian flow $\Phi_t$ \emph{contracts the
symplectic form exponentially} with respect to the damping coefficient $\gamma$.
It follows from \eqref{sym_conf} that
volumes on phase space shrink as
$\vol ( \Phi_t(\mathcal{R}) )  = \int_{\mathcal{R}} | \det J_t(z)| dz
  = e^{-n \gamma t} \vol (\mathcal{R})
$
where $\mathcal{R} \subset \mathbb{R}^{2n}$.
This contraction
is stronger as dimension increases.
The conservative case is recovered with $\gamma = 0$ above; in this case, the symplectic structure
is preserved and volumes remain invariant (Liouville's theorem).
A known and interesting property of conformal Hamiltonian systems is that their
Lyapunov exponents sum up in pairs to $\gamma$ \cite{Dessler:1988}.
This imposes constraints on the admissible dynamics and controls the phase
portrait near critical points.
For other properties of attractor sets we refer to \cite{Maro:2017}.
Finally, conformal symplectic transformations can be composed
and form the so-called conformal group.

\section{Conformal symplectic optimization}
\label{conformal_symplectic}


Consider \eqref{confhamuni}
where we associate flows $\Phi^{C}_t$ and $\Phi^{D}_t$ to the respective
vector fields $C(z)$ and $D(z)$.
Conformal symplectic integrators can be constructed as \emph{splitting methods} that
approximate the true flow $\Phi_t$ by composing the individual
flows
$\Phi_t^C$ and $\Phi_t^D$.
Our procedure to obtain a numerical map $\Psi_h$, with
step size $h > 0$,
is to first obtain a numerical approximation to the
conservative part of the system,
$\dot{z} = \Omega \nabla H(z)$. This yields a numerical map $\Psi^C_h$
that approximates $\Phi^C_h$ for small intervals of time $[t,t+h]$.
One can choose any standard \emph{symplectic integrator} for this task.
Let us pick the simplest, i.e.  the
symplectic Euler method \cite[pp. 189]{Hairer}. We thus have
$\Psi_h^{C}: (x,p) \mapsto (X,P)$ where
\begin{equation}
\label{phic}
X = x + h \nabla_p H(x, P), \qquad
P = p - h \nabla_x H(x, P).
\end{equation}
Now the dissipative part of the system, $\dot{z} = -\gamma D z$,
can be integrated exactly. Indeed,
$\dot{x} = 0$ and $\dot{p} = -\gamma p$, thus
$\Psi_h^D : (x,p) = (x, e^{-\gamma h} p)$.
With
$\Psi_h \equiv \Psi_h^C \circ \Psi_h^D$ we obtain
$\Psi_h : (x, p) \mapsto (X, P)$ as
\begin{equation}
  \label{gen_euler}
  P = e^{-\gamma h} p - h \nabla_x H(x,P), \qquad
  X = x + h \nabla_p H(x, P) .
\end{equation}
This is nothing but a \emph{dissipative version} of the symplectic Euler method.
Similarly, if we choose the leapfrog method \cite[pp. 190]{Hairer} for  $\Psi_h^C$ and consider $\Psi_h \equiv \Psi_{h/2}^D \circ \Psi_h ^C \circ \Psi_{h/2}^D$ we obtain
\begin{subequations}\label{gen_leap}
\begin{align}
\tilde{X} &= x + \dfrac{h}{2} \nabla_p H\big(\tilde{X}, e^{-\gamma h/2} p\big), \label{gen_leap1} \\
\tilde{P} &= e^{-\gamma h/2} p - \dfrac{h}{2}\big( \nabla_x H\big(\tilde{X}, e^{-\gamma h/2}p\big) + \nabla_x H\big(\tilde{X}, \tilde{P}\big)\big), \\
X &= \tilde{X} + \dfrac{h}{2} \nabla_p H(\tilde{X}, \tilde{P}), \\
P &= e^{-\gamma h / 2}\tilde{P}. \label{gen_leap4}
\end{align}
\end{subequations}
This is a \emph{dissipative} version of the leapfrog, which is recovered when $\gamma = 0$.
Note that in general \eqref{gen_euler} is implicit in $P$, and \eqref{gen_leap} is implicit in $\tilde{X}$ and $P$. However,
both will become explicit for separable Hamiltonians, $H = T(p) + f(x)$, and in this case they are extremely efficient.
Note also that \eqref{gen_euler} and \eqref{gen_leap} are completely general, i.e. by choosing
a suitable Hamiltonian $H$ one can obtain several possible optimization algorithms from these integrators.
Next, we show important
properties of these integrators.
(Below we denote $t_k = k h$ for $k=0,1,\dotsc$, $z_k \equiv z(t_k)$, etc.)



\begin{definition}[Order of accuracy]\label{order}
A numerical map
$\Psi_h$
is said to be of order $r\ge 1$ if
$\| \Psi_h(z) - \Phi_h(z)\| = O(h^{r+1})$
for any $z \in \mathbb{R}^{2n}$. (Recall that  $h > 0$ is the step size and $\Phi_h$ is the true flow.)
\end{definition}

\begin{definition}[Conformal symplectic integrator]\label{confmap}
A numerical map
$\Psi_h$
is said to be conformal symplectic if
$z_{k+1} = \Psi_h(z_k)$ is
conformal symplectic, i.e.
$\omega_{k+1} = e^{-\gamma h} \omega_k$,
whenever $\hat{\Phi}_h$ is applied to a smooth
Hamiltonian. Iterating such a map yields
$\omega_k =
e^{-\gamma t_k} \omega_0$ so that \eqref{sym_conf} is preserved.
\end{definition}


\begin{theorem}  \label{confsymp}
Both methods \eqref{gen_euler} and \eqref{gen_leap} are conformal symplectic.
\end{theorem}
\begin{proof}
Note that in both cases $\Psi_h^C$ is a symplectic
integrator, i.e. its Jacobian $J^C_h$ obeys $(J^C_h)^T \Omega J_h^C = \Omega$---see \eqref{sym_conf} with $\gamma = 0$.
Now the map $\Psi_h^D$ defined above is conformal symplectic, i.e. one can verify
that its Jacobian $J^D_h$ obeys $(J^D_h)^T \Omega J^D_h = e^{-\gamma h} \Omega$.
Hence, any composition of these maps will be conformal symplectic. For instance,
\begin{equation}
  (J^C_h J^D_h)^T \Omega (J^C_h J^D_h) = (J^{D}_h)^T (J^C_h)^T \Omega J^C_h J^D_h = (J^D_h)^T \Omega J^D_h = e^{-\gamma h} \Omega.
\end{equation}
The same would be true for any type of composition whose overall time step add up to $h$.
\end{proof}

\begin{theorem}\label{accuracy}
The numerical scheme \eqref{gen_euler} is of order $r=1$, while \eqref{gen_leap} is of order $r=2$.
\end{theorem}
\begin{proof}
The proof simply involves manipulating Taylor expansions for the numerical method and for the continuous-time
system over a time interval of $h$; this is presented in Appendix~\ref{proof_accuracy}.
\end{proof}

We mention that one can construct higher order integrators by following the above approach, however these
would be more expensive, involving more gradient computations per iteration.
In practice, methods of order $r=2$ tend to have the best cost benefit.

\section{Symplectic structure of heavy ball and Nesterov}
\label{momentum-nesterov}

Consider the classical Hamiltonian \eqref{hamiltonian1} and replace into
\eqref{gen_euler} to obtain
\begin{equation}
\label{simpleap}
p_{k+1}  = e^{-\gamma h} p_k - h \nabla f(x_k),
\qquad
x_{k+1} = x_k + \tfrac{h}{m} p_{k+1},
\end{equation}
where we now make the iteration number $k=0,1,\dotsc$ explicit for convenience of the reader in relating
to optimization methods. Introducing a change of variables,
\begin{equation}
\label{momgd_params}
v_k \equiv \dfrac{h}{m} p_k, \qquad \epsilon \equiv \dfrac{h^2}{m}, \qquad \mu \equiv e^{-\gamma h},
\end{equation}
we  see that \eqref{simpleap} is precisely the well-known CM method \eqref{momgd2}.
Therefore, CM is nothing but a dissipative version of the symplectic Euler method.
Thanks to Theorems~\ref{confsymp}~and~\ref{accuracy} we have:

\begin{corollary}[CM is ``symplectic'']\label{momgd_symp}
The classical momentum or heavy ball method \eqref{momgd2} is a conformal symplectic integrator for the Hamiltonian
system \eqref{hamiltonian1}. Moreover, it is an integrator of order $r=1$.
\end{corollary}

Consider again the Hamiltonian \eqref{hamiltonian1} but replaced into \eqref{gen_leap}.
Let us also replace the last update \eqref{gen_leap4}, i.e. from a previous iteration, into the first update \eqref{gen_leap1}.\footnote{Note that it is valid to replace successive updates without changing the algorithm.} We thus obtain
\begin{equation} \label{leapclass}
x_{k+1/2} = x_{k} + \dfrac{h}{2m} e^{-\gamma h} p_k, \quad
p_{k+1} = e^{-\gamma h} p_k - h \nabla f(x_{k+1/2}), \quad
x_{k+1} = x_{k+1/2} + \dfrac{h}{2m} p_{k+1}.
\end{equation}
Define
\begin{equation} \label{nag_params}
v_{k} \equiv \dfrac{h}{2m} p_k, \qquad \epsilon \equiv \dfrac{h^2}{2m}, \qquad \mu \equiv e^{-\gamma h}.
\end{equation}
Then  \eqref{leapclass} can be written as
\begin{equation} \label{almost_nag}
x_{k+1/2} = x_{k} + \mu v_k, \qquad
v_{k+1} = \mu v_k - \epsilon \nabla f(x_{k+1/2}), \qquad
x_{k+1} = x_{k+1/2} + v_{k+1}. 
\end{equation}
The reader can immediately recognize the close similarity with NAG  \eqref{nag2};
this would be exactly NAG if we replace $x_{k+1/2} \to x_{k}$ in the third update above.
As we will show next, this small difference has actually profound consequences.
Intuitively, by ``rolling this last update backwards'' one introduces a spurious friction
into the method, as we will show
through a symplectic perspective (Theorem~\ref{nag_contract_omega} below).
The method \eqref{leapclass}  is actually a second order accurate
version of \eqref{simpleap}.
In order to analyze the symplectic structure one must work on the phase space $(x,p)$.  The true phase space
equivalent to NAG is given by
\begin{subequations} \label{nag_mom}
\begin{align}
  x_{k+1/2} &= x_{k} + \dfrac{h}{m} e^{-\gamma h} p_k, \label{nag_mom1}\\ 
  p_{k+1} &= e^{-\gamma h} p_k - h \nabla f(x_{k+1/2}), \\ 
  x_{k+1} &= x_{k} + \dfrac{h}{m} p_{k+1},
\end{align}
\end{subequations}
which is completely equivalent to \eqref{nag2} under  the correspondence \eqref{momgd_params}.

\begin{theorem}[NAG is not ``symplectic''] \label{nag_contract_omega}
Nesterov's accelerated gradient \eqref{nag2}, or equivalently \eqref{nag_mom}, is an integrator of order $r=1$ to the Hamiltonian system \eqref{hamiltonian1}. This method is not conformal symplectic but  rather contracts the symplectic form as
\begin{equation} \label{nest_omega}
\omega_{k+1} =
e^{- \gamma h} \left[ I - \dfrac{h^2}{m}  \nabla^2 f(x_k) \right]
\omega_k  + O(h^3).
\end{equation}
\end{theorem}
\begin{proof}
We work on  phase space variables $(x,p)$ thus NAG should be considered in the form \eqref{nag_mom}.
First we derive the order of accuracy of this method with respect to its underlying 
Hamiltonian system:
\begin{equation} \label{class_harm}
\dot{x} = \dfrac{p}{m}, \qquad \dot{p} = - \nabla f(x) - \gamma p.
\end{equation}
Denote $x = x(t_k)$ and $p = p(t_k)$ and expand the  exponential in \eqref{nag_mom1} to obtain
$x_{k+1/2} = x + \tfrac{h}{m} p - \tfrac{h^2}{m} \gamma p + O(h^3)$.
Using this and  Taylor expansions in the last two updates of \eqref{nag_mom} yield
\begin{subequations} \label{nag_exp}
\begin{align}
p_{k+1} &= p - h \gamma p -h \nabla f(x) + \dfrac{h^2}{2} \gamma^2 p - \dfrac{h^2}{m}\nabla^2 f(x) p_k + O(h^3), \\
x_{k+1} &= x + \dfrac{h}{m}p - \dfrac{h^2}{m} \gamma p - \dfrac{h^2}{m} \nabla f(x) + O(h^3),
\end{align}
\end{subequations}
where it is implicit that $\nabla f$ and $\nabla^2 f$ are computed at $(x,p)$.
From \eqref{class_harm} we readily have
\begin{subequations} \label{nag_cont_exp}
\begin{align}
p(t_k+h) &= p - h \nabla f - h\gamma p - \dfrac{h^2}{2m} \nabla^2 f p
  +\dfrac{h^2}{2} \gamma \nabla^2 f p + \dfrac{h^2}{2} \gamma \nabla f + \dfrac{h^2}{2}\gamma^2 p + O(h^3), \\
  x(t_k+h) &= x + \dfrac{h}{m} p -\dfrac{h^2}{2m} \gamma p + O(h^3).
\end{align}
\end{subequations}
Hence, by comparison with \eqref{nag_exp} we conclude that
$x_{k+1} = x(t_k+h) + O(h^2)$ and $p_{k+1} = p(t_k + h) + O(h^2)$, which according to Definition~\ref{order} means that NAG is an integrator of order $r=1$.

Second, we investigate how NAG deforms the symplectic structure.
Consider the variational form of \eqref{nag_mom} (the notation is standard \cite{Hairer}):
\begin{subequations}
\label{nagsymp}
\begin{align}
dx_{k+1/2}&=dx_k + \dfrac{h}{m}e^{-\gamma h} dp_k, \\
dp_{k+1}&=e^{-\gamma h}dp_k - h \nabla^2f(x_{k+1/2})dx_{k+1/2},  \\
dx_{k+1} &= dx_k + \dfrac{h}{m} dp_{k+1}.
\end{align}
\end{subequations}
Using these, bilinearity and the antisymmety of the wedge product, together the fact that $\nabla^2 f$ is symmetric, we obtain
\begin{equation}
\label{nag_not_symp}
\begin{split}
dx_{k+1} \wedge d p_{k+1} &= dx_{k} \wedge dp_{k+1} \\
&= e^{-\gamma h } dx_{k} \wedge dp_k -
h  dx_{k} \wedge \nabla^2 f\big(x_{k+1/2}) dx_{k+1/2} \\
&= e^{-\gamma h } dx_{k} \wedge dp_k -\dfrac{h^2}{m} e^{-\gamma h} dx_{k} \wedge \nabla^2f(x_{k+1/2}) dp_k \\
&= e^{-\gamma h } dx_{k} \wedge dp_k -\dfrac{h^2}{m} e^{-\gamma h} dx_{k} \wedge \nabla^2f(x_{k}) dp_k
+ O(h^3),
\end{split}
\end{equation}
where in the last passage we used a Taylor approximation for $x_{k+1/2}$.
  Thus, $dx_{k+1} \wedge dp_{k+1} \ne e^{-\gamma h} dx_{k}\wedge dp_{k}$, showing
  that the method is not conformal symplectic (see Definition~\ref{confmap}). Moreover, using the symmetry of $\nabla^2 f$ we can write \eqref{nag_not_symp} as \eqref{nest_omega}.
\end{proof}

While CM exactly preserve the same dissipation found in the underlying 
continuous-time system, NAG
introduces some extra contraction or expansion of the symplectic form, depending whether $\nabla^2 f$ is positive
definite or not.
  From \eqref{nest_omega}, in $k$ iterations of NAG, and neglecting the $O(h^3)$ error term, we have
  \begin{equation}
    \begin{split}
    \omega_k &\approx e^{-\gamma t_k}
    \prod_{i=1}^k \left[ I - \dfrac{h^2}{m} \nabla^2 f(x_{k-i}) \right] \omega_0  \\
    &\approx e^{-\gamma t_k} \left[ I - \dfrac{h^2}{m}\big(\nabla^2 f(x_{k-1}) - \nabla^2f(x_{k-2}) - \dotsm - \nabla^2 f(x_0)\big) \right] \omega_0 .
  \end{split}
  \end{equation}
  This depends on the entire history of the Hessians from the initial point.
  Therefore, NAG contracts the symplectic form slightly more than the underlying conformal Hamiltonian system---assuming $\nabla^2 f$ is positive definite---and
  it does so in a way that depends on the Hessian of the objective function. Note that this is a small effect of $O(h^2)$.
  Moreover, if $\nabla^2 f$ has negative eigenvalues, e.g. $f$ is nonconvex and has saddle points, then NAG actually introduces some
  spurious excitation in that direction.
  To gain some intuition, consider the simple case of a quadratic function:\footnote{
  This quadratic function  is actually enough to capture
  the behaviour when close to a critical point $x^\star$ since $f(x) \approx f(x^\star) + \tfrac{1}{2} \nabla^2 f(x^\star) x$
  and one can work on rotated coordinates where $\nabla^2 f(x^\star) = \diag(\lambda_1, \dotsc, \lambda_n)$.
  }
  \begin{equation} \label{quad1d}
    f(x) = (\lambda /2) x^2
  \end{equation}
  for some constant $\lambda$. Thus \eqref{nest_omega} becomes
  \begin{equation} \label{nag_contract_quad}
    \omega_{k+1} \approx e^{-\gamma h + \log(1 - h^2\lambda / m)} \omega_k \approx e^{-(\gamma + h \lambda / m)h  } \omega_k
    \quad \implies \quad \omega_k \approx e^{-(\gamma + h\lambda/m) t_k} \omega_0.
  \end{equation}
  This suggests that, effectively, the original damping of the system is being replaced by $\gamma \to \gamma + h \lambda / m$. Thus, if $\lambda > 0$ there is  some spurious damping, whereas if $\lambda < 0$ there is some spurious excitation.
  We will confirm this conclusion from another perspective
  in Section~\ref{proof_nag_contract_omega} below.

\subsection{Alternative form}
It is perhaps more common to find Nesterov's method in the following form \cite{Nesterov:1983}:
\begin{equation}
\label{nag3}
x_{k+1} = y_k -\epsilon \nabla f (y_k), \qquad
y_{k+1} = x_{k+1} + \mu_{k+1} (x_{k+1}-x_k),
\end{equation}
where $\mu_{k+1} = k/(k+3)$.
This is equivalent to \eqref{nag2}, as can be seen by
introducing the variable
$v_k \equiv x_k - x_{k-1}$
and writing the updates
in terms of $x$ and $v$.
When $\mu_k$ is constant,
Theorem~\ref{nag_contract_omega} shows that the method is not conformal symplectic.
When $\mu_k = k/(k+3)$, the differential equation associated
to \eqref{nag3} is equivalent to
\eqref{confhameq}/\eqref{hamiltonian1} with  $\gamma = 3 / t$.
It is possible to generalize the above results for time dependent cases \cite{Franca:2020}.
Therefore, also in this case, NAG does not preserve the symplectic structure; we note that
\eqref{nest_omega} still holds with $e^{-\gamma h} \to e^{-3 \log(1 + h/t_k )}$ where $t_k=h k$.

\subsection{Preserving stability and continuous-time rates}
An important question is whether being ``symplectic'' is beneficial or not for optimization.
Very recently, it has been shown \cite{Franca:2020} that symplectic discretizations of dissipative systems may indeed preserve
continuous-time rates of convergence when $f$ is smooth and the system
is appropriately dampened (choice of $\gamma$); the continuous-time rates can be
obtained via  Lyapunov analysis.
Thus, assuming that we have a suitable conformal Hamiltonian system, conformal symplectic integrators such as the general method \eqref{gen_leap},
provide a principled approach to construct
optimization algorithms that are guaranteed to respect the main properties
of the system, such as stability of critical points and convergence rates.
Furthermore,
we claim that there is a delicate tradeoff where being conformal symplectic is related to an improved stability, in the sense that the method can operate with larger step sizes, while
the spurious dissipation introduced by NAG (Theorem~\ref{nag_contract_omega})
may improve the convergence rate slightly,
since it introduces more contraction, but at the cost of making the method less stable; we show these details in Section~\ref{stability_appendix}.
Next, we also provide important additional insights into
CM and NAG, such as their  modified or perturbed equations
and their \emph{shadow Hamiltonians},
which describe these methods
to a higher degree of resolution.


\subsection{Shadow dynamical systems for Nesterov and heavy ball}  \label{proof_nag_contract_omega}

We have shown above that both CM and NAG are a first order integrators to the conformal Hamiltonian system \eqref{class_harm}, however NAG
changes slightly the behaviour of the original system since it introduces spurious damping or excitation.
To understand its behaviour more closely, one can ask the following question: \emph{for which continuous-time dynamical system
NAG turns out to be a second order integrator?}
In other words, we can look for a modified system that captures the behaviour of NAG more closely, up to $O(h^3)$.
Every numerical method is known to have a modified or perturbed differential equation \cite{Hairer} (the brief discussion in \cite{Franca:2020} may also be useful).
In answering this question, we thus find the following.

\begin{theorem}[Shadow dynamical system for Nesterov's method] \label{nag_shadow}
NAG \eqref{nag2}, or its equivalent phase space representation \eqref{nag_mom},
is a second order integrator to the following modified or perturbed equations:
\begin{equation} \label{mod_eq_nag}
  \dot{x} = \dfrac{1}{m} p - \dfrac{\gamma h }{2 m} p -\dfrac{h}{2m} \nabla f(x) , \qquad
  \dot{p} = -\nabla f(x) -\gamma p -\dfrac{h \gamma}{2} \nabla f - \dfrac{h}{2m} \nabla^2 f(x) p.
\end{equation}
\end{theorem}
\begin{proof}
We look for vector fields $F(q,p;h)$ and $G(q,p;h)$ for the modified system
\begin{equation} \label{mod1}
  \dot{x} = \dfrac{p}{ m}  + h F, \qquad \dot{p} = -\nabla f(x) - \gamma p + h G,
\end{equation}
such that \eqref{nag_mom} is an integrator of order $r=2$. This can be done by computing \cite{Hairer}
\begin{equation}
F = \lim_{h \to 0} \dfrac{x_{k+1} - x(t_k + h)}{h^2}, \qquad
G = \lim_{h \to 0} \dfrac{p_{k+1} - p(t_k + h)}{h^2}.
\end{equation}
From \eqref{nag_exp} and \eqref{nag_cont_exp} we obtain precisely
\eqref{mod_eq_nag}.  By the previously discussed approach through Taylor expansions
one can also readily check that NAG is indeed an integrator
of order $r=2$ to this perturbed system.
\end{proof}

Note that we can combine  \eqref{mod_eq_nag} into a second order differential equation:
\begin{equation}\label{mod_eq_single_nag}
m \ddot{x} + m \left( \gamma I + \dfrac{h}{m} \nabla^2 f(x)\right) \dot{x} =
-\left( I + \dfrac{h \gamma}{2} I - \dfrac{h^2 \gamma^2}{4} I + \dfrac{h^2}{4 m} \nabla^2 f(x) \right) \nabla f(x),
\end{equation}
where $I$ denotes the $n\times n$ identity matrix.
We see that this equation has several new ingredients compared to
\begin{equation} \label{sec_eq}
\ddot{x} + \gamma \dot{x} = -(1/m) \nabla f(x),
\end{equation}
 which is equivalent to \eqref{class_harm}. First, when $h \to 0$ the system \eqref{mod_eq_single_nag}
 recovers \eqref{sec_eq}, as it should
since both must agree to leading order. Second,  the spurious change in the damping coefficient  reflects
the behaviour of the symplectic form \eqref{nest_omega} (see also \eqref{nag_contract_quad}). Third,  we  see
that the gradient $\nabla f$ is rescaled by the contribution of  several terms, including the Hessian $\nabla^2 f$, making explicit
a curvature dependent behaviour, which also appears in the damping coefficient.
Note that the modified equation \eqref{mod_eq_single_nag}, or equivalently \eqref{mod_eq_nag},
depends on the step size $h$, hence it captures an intrinsic behaviour of the discrete-time algorithm that is not captured by \eqref{class_harm}.


Since CM is also a first order integrator to \eqref{class_harm}, which is actually conformal symplectic,
it is natural to consider its modified equation and compare with the one for NAG \eqref{mod_eq_nag}.
We thus obtain the following.

\begin{theorem}[Shadow Hamiltonian for heavy ball] \label{cm_shadow}
The heavy ball or CM method \eqref{momgd2}, equivalently written in the phase space as \eqref{simpleap},
is a second order integrator to the following modified conformal Hamiltonian system:
\begin{equation} \label{mod_eq_cm}
  \dot{x} = \dfrac{1}{m} p - \dfrac{h \gamma}{2m} p - \dfrac{h}{2m} \nabla f(x), \qquad
  \dot{p} = - \nabla f(x) - \gamma p - \dfrac{h \gamma}{2} \nabla f(x) + \dfrac{h}{2m} \nabla^2 f(x) p .
\end{equation}
Such a system admits the shadow Hamiltonian
\begin{equation}
\label{shadow_ham_cm}
\tilde{H} = \dfrac{1}{2m}\| p\|^2  + f(x) - \dfrac{h \gamma }{4m}\| p \|^2 - \dfrac{h}{2m} \langle \nabla f(x), p \rangle + \dfrac{h \gamma}{2} f .
\end{equation}
\end{theorem}
\begin{proof}
It follows exactly as in Theorem~\ref{nag_shadow}. Also, one can readily  verify that replacing \eqref{shadow_ham_cm} into
\eqref{confhameq} gives \eqref{shadow_ham_cm}.
\end{proof}

We note the striking similarity between \eqref{mod_eq_cm} and \eqref{mod_eq_nag};
the only difference is the sign of the last term in the second equation.
Up to this level of resolution, the difference is that NAG introduces a spurious damping compared to CM, in agreement
with the derivation of the symplectic form \eqref{nest_omega}.
On the other hand, notice that the perturbed system \eqref{mod_eq_cm} for CM is conformal Hamiltonian, contrary to \eqref{mod_eq_nag} that cannot be written in Hamiltonian form;   this is the reason why structure-preserving discretizations tend to be more stable, since the perturbed trajectories are always close, i.e. within a bounded error, from the original Hamiltonian dynamics.
We can also combine \eqref{mod_eq_cm} into
\begin{equation}\label{mod_eq_single_cm}
m \ddot{x} + m \gamma \dot{x} =
- \left( I + \dfrac{h \gamma}{2} I  - \dfrac{h^2 \gamma^2}{4} I - \dfrac{h^2}{4 m} \nabla^2 f(x) \right) \nabla f(x).
\end{equation}
Again, this is strikingly similar to \eqref{mod_eq_single_nag}. Note that this equation does not have the
spurious damping term $(h/m) \nabla^2 f(x)$ as in \eqref{mod_eq_single_nag}, making even more explicit that it preserves exactly the dissipation of the original continuous-time system.   As we will show later, there is a balance between preserving such a dissipation and the stability of the method.  While NAG introduces an extra damping, and may slightly help in an improved convergence since it dissipates more energy, this comes at the price in a decreased stability. Before showing this explicitly in Section~\ref{stability_appendix}, we first introduce a new optimization methods based on a relativistic system.

\section{Dissipative relativistic optimization}
\label{relativistic}

Let us briefly
mention some simple but fundamental concepts to motivate
our approach.
The previous algorithms
are based on
\eqref{hamiltonian1} which leads to a classical Newtonian system where
time is just a parameter, independent
of the Euclidean space where the trajectories live. This implies that
there is no restriction on the speed,
$\|v\|  = \|dx/dt\|$, that a particle can attain. This
translates to a discrete-time algorithm, such as
\eqref{simpleap}, where
large gradients $\nabla f$ give rise to a large momenta
$p$, implying that the position
updates for $x$ can diverge.
On the other hand, in special relativity,
space and time
form a unified geometric entity,
the $(n+1)$-dimensional Minkowski
spacetime with
coordinates $X = (ct; x)$,
where $c$ denotes the speed of light.
An infinitesimal distance on this manifold is given by
$ds^2= -(c dt)^2 + \| dx \|^2$.
Null geodesics correspond
to $ds^2 = 0$, implying
$\| v\|^2 = \| dx/dt\|^2 = c^2$, i.e.
no particle can travel faster than $c$.
This imposes constraints
on the geometry where trajectories take place---it is actually a hyperbolic geometry.
With that being said,
the  idea is that by discretizing
a relativistic system we can incorporate these features into an
optimization algorithm which may  bring benefits such as
an improved stability.

A relativistic  particle subject to a potential $f$
is described by the following Hamiltonian:
\begin{equation}
  \label{relham}
  H(x,p) = c \sqrt{\| p \|^2 + m^2 c^2 } + f(x).
\end{equation}
In the classical limit, $\| p\| \ll m c$, one obtains
$H = mc^2 + \| p\|^2/(2m) + f(x) + O(1/c^2)$, recovering
\eqref{hamiltonian1} up to the  constant $E_0 = mc^2$, which has no effect
in deriving the equations of motion.
Replacing \eqref{relham}
into~\eqref{confhameq}
we thus obtain a \emph{dissipative relativistic system}:
\begin{equation}
\label{confrel}
\dot{x} = \dfrac{c  p}{  \sqrt{  \| p\|^2 + m^2 c^2 } } , \qquad
\dot{p} = -\nabla f - \gamma p.
\end{equation}
Importantly, in \eqref{confrel}
the momentum is normalized by the $\sqrt{\cdot}$ factor,
so  $\dot{x}$ remains bounded even if $p$ was to go unbounded.
Now, replacing \eqref{relham} into the first order accurate conformal symplectic
integrator \eqref{gen_euler}, we readily obtain
\begin{equation}
\label{rgd0}
p_{k+1} = e^{-\gamma h}  p_k - h \nabla f(q_k),  \qquad
x_{k+1} = x_k +
\dfrac{h c  p_{k+1} }{ \sqrt{   \| p_{k+1}\|^2 + m^2 c^2 } }.
\end{equation}
When $c \to \infty$ the above updates recover CM \eqref{simpleap}. Thus, this method is a relativistic generalization
of CM or heavy ball. Moreover, the method  \eqref{rgd0} is a first order
conformal symplectic integrator by construction (see Theorems~\ref{confsymp} and \ref{accuracy}).

One can replace the Hamiltonian \eqref{relham} into \eqref{gen_leap} to obtain
a second order version of \eqref{rgd0}.  However, motivated by the close connection between NAG
and \eqref{leapclass}---recall the comments following \eqref{almost_nag} about NAG ``rolling back'' the last update---let us additionally
introduce a convex
combination, $\alpha x_{k+1/2} + (1-\alpha) x_k$ where $0\le \alpha \le 1$, between the initial and  midpoint
of the method.  In this manner, we can interpolate between a conformal symplectic regime and
  a spurious Hessian damping regime (recall Theorem~\ref{nag_contract_omega}).
Therefore, we obtain the following integrator:
\begin{subequations} \label{rgd}
\begin{align}
x_{k+1/2} &= x_k + (h c/2)  e^{-\gamma h/2} p_k \Big/ \sqrt{  e^{-\gamma h} \| p_k\|^2 + m^2 c^2} , \\
p_{k+1/2} &= e^{-\gamma h/2} p_k - h \nabla f(x_{k+1/2}), \\
x_{k+1} &= \alpha x_{k+1/2} + (1-\alpha)x_{k} + (h c /2) p_{k+1/2} \Big/ \sqrt{  \| p_{k+1/2}\|^2 +m^2 c^2 }, \\
p_{k+1} &= e^{-\gamma h/2} p_{k+1/2}.
\end{align}
\end{subequations}
We call this method \emph{Relativistic Gradient Descent} (RGD). By introducing
\begin{equation}
  v_k \equiv \dfrac{h}{2m} p_k, \qquad \epsilon \equiv \dfrac{h^2}{2m},
  \qquad \mu \equiv e^{-\gamma h}, \qquad
  \delta \equiv \dfrac{4}{ c^2 h^2},
\end{equation}
the updates \eqref{rgd} assume the equivalent form stated in Algorithm~\ref{rgd1} in the introduction.

RGD \eqref{rgd} (resp. Algorithm~\ref{rgd1}) has several interesting limits,
recovering the behaviour of known algorithms as particular cases.
For instance,
when $c \to \infty$ (resp. $\delta \to 0$) it reduces to an interpolation
between CM \eqref{simpleap} (resp. \eqref{momgd2}) and NAG \eqref{nag_mom} (resp. \eqref{nag2}).
If we additionally set $\alpha = 0$ it becomes precisely NAG, whether when $\alpha=1$ it becomes
a second order version (in terms of accuracy) of CM.\footnote{The dynamics of both CM and this second order version is pretty close, and if anything the latter is even more stable than the former (see Section~\ref{stability_appendix}).}
When $\alpha = 1$, and arbitrary $c$ (or $\delta$), RGD is a conformal symplectic integrator
thanks to Theorems~\ref{confsymp}. Recall also that Theorem~\ref{accuracy}  implies that RGD
is a second order accurate integrator.
When $\alpha =0 $, and arbitrary $c$ (or $\delta$), RGD is no longer conformal symplectic and introduces
a Hessian driven damping in the spirit of NAG.
Finally, the parameter $c$ (or $\delta$) controls the strength of the normalization
term in the position updates of \eqref{rgd} (or Algorithm~\ref{rgd1}), which can help preventing divergences
when navigating through a rough landscape with large gradients, or fast growing tails. Indeed, note that
$\| x_{k+1} - \alpha x_{k+1/2} - (1-\alpha) x_k \| \le 1/\delta$ is always bounded for $\delta > 0$;
this becomes unbounded when $\delta \to 0$, i.e. in the classical limit of CM and NAG.

In short, RGD is a novel algorithm  with quite some flexibility and unique features,
generalizing perhaps the two most important accelerated gradient based  methods in the literature, which can be recovered as limiting cases.  Next, we illustrate numerically through
simple yet insightful examples  that RGD can
be more stable and faster than CM and NAG.

\section{Tradeoff between stability and convergence rate} \label{stability_appendix}

Here we illustrate an interesting phenomenon: there is a tradeoff between stability versus
convergence rate.  Intuitively, an improved rate is associated to a higher ``contraction,'' i.e. the introduction
of spurious dissipation in the numerical method.   However, this makes the method less stable, and ultimately
very sensitive to parameter tuning.  On the other hand, a geometric or structure-preserving integrator may
have slightly less contraction, since it preserves the original dissipation of the continuous-time system exactly,
but it is more stable and able to operate with larger step sizes.  Furthermore, a structure-preserving
method is guaranteed to reproduce very closely, perhaps even up to a negligible error, the continuous-time rates of convergence \cite{Franca:2020}.  This indicates that there may have benefits in considering this class of
methods for optimization, such as conformal symplectic integrators that are being advocated in this paper.

Stability of a numerical
integrator means the region of hyperparameters, e.g. values of the step size, such that the method is able to converge.
The larger this region, more stable is the method.
The convergence rate is a measure of how fast the method tends to the minimum, and
this is related to the amount of contraction between subsequent states, or subsequent values of the objective function.  For instance,
since NAG introduces some spurious dissipation---recall \eqref{nest_omega} and \eqref{mod_eq_single_nag}---we
expect that it may have a slightly higher contraction compared to CM, which exactly
preserves the dissipation of the continuous-time system---recall \eqref{mod_eq_single_cm}.
Thus, such a spurious dissipation can induce a slightly improved convergence rate, but  as we
will show below, at the cost of making the method more unstable and thus requiring smaller step sizes.

Let us consider a standard linear stability analysis, which involves
a quadratic function \eqref{quad1d} such that the previous methods can be treated analytically.
Thus, replacing \eqref{quad1d} into CM in the form \eqref{simpleap} it is possible to write the algorithm
as a linear system:
\begin{equation}
  z_{k+1} = T_{\textnormal{CM}} z_k, \qquad
  T_{\textnormal{CM}} = \begin{bmatrix} 1 - h^2 \lambda / m & (h/m) e^{-\gamma h} \\ -h \lambda & e^{-\gamma h}\end{bmatrix},
\end{equation}
where we denote $z = \big[ \begin{smallmatrix} x \\ p \end{smallmatrix} \big]$.
Similarly, NAG in the form \eqref{nag_mom} yields
\begin{equation}
  z_{k+1} = T_{\textnormal{NAG}} z_k, \qquad
  T_{\textnormal{NAG}} = \begin{bmatrix} 1 - h^2 \lambda / m & (h/m) e^{-\gamma h}(1-h^2\lambda/m) \\ -h \lambda & e^{-\gamma h}(1-h^2\lambda / m)\end{bmatrix},
\end{equation}
while RGD \eqref{rgd}, with $c\to \infty$ and $\alpha = 1$, yields\footnote{The case of finite $c$ is nonlinear and not amenable to such an analysis. However, the case $c\to \infty$ already provides useful insights.}
\begin{equation}
  z_{k+1} = T_{\textnormal{RGD}} z_k, \qquad
  T_{\textnormal{RGD}} = \begin{bmatrix} 1 - h^2 \lambda / (2m) & h/(2m) e^{-\gamma h/2}(2-h^2\lambda/(2m)) \\ -h \lambda e^{-\gamma h /2} & e^{-\gamma h}(1-h^2\lambda / (2m))\end{bmatrix}.
\end{equation}
A linear system is stable if  the spectral radius of its transition matrix is $\rho(T) \le 1$. We can  compute
the eigenvalues of the above matrices and check for which range of parameters they remain inside the unit circle; e.g.
for  given $\gamma$, $m$, and $\lambda$ we can find the allowed range of the step size $h$ for which the maximum eigenvalue in absolute value  is $ |\lambda_{\textnormal{max}}| \le 1$.
Instead of showing the explicit formulas for these eigenvalues, which can be obtained quite simply but are cumbersome, let us illustrate what happens graphically.

\begin{figure}[t]
\centering
\includegraphics[width=\textwidth]{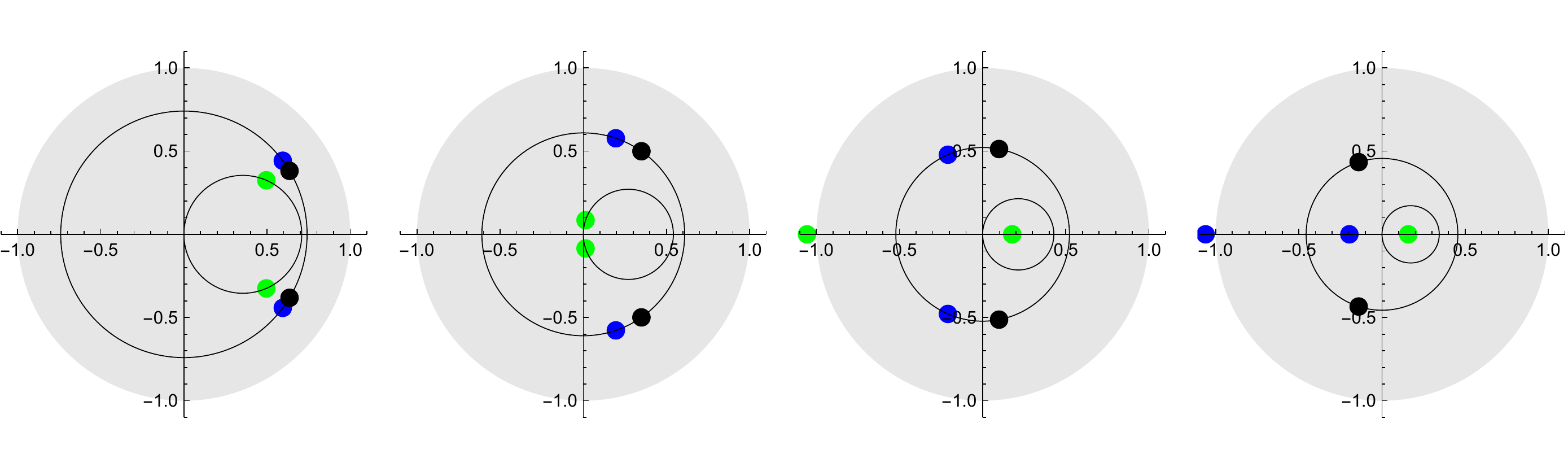}
\caption{\label{circles}
Stability of CM \eqref{simpleap} (\emph{blue}), NAG \eqref{nag_mom} (\emph{green}), and RGD \eqref{rgd} with $c\to \infty$ and $\alpha=1$ (\emph{black})---in this case it becomes a dissipative version of the Leapfrog to
system \eqref{class_harm}. We plot the eigenvalues in the complex plane; $x$-axis is the real part, $y$-axis
is the imaginary part.
The unit circle represent the stability region, i.e. once an
eigenvalue leaves the gray area the corresponding method becomes unstable. Both CM and RGD
are symplectic thus their eigenvalues always move on a circle of radius $e^{-\gamma h / 2}$
centered at the origin. NAG has eigenvalues in the smaller circle with radius
$1/(e^{\gamma h}+1)$ and centered at $1/(e^{\gamma h}+1)$ on the $x$-axis; the circle is
dislocated from the origin precisely due to spurious dissipation.
From left to right we increase the step size $h$ while keeping $\gamma$, $m$, and $\lambda$ fixed.
As $h$ increases the eigenvalues move on the circles in the counterclockwise direction until they fall on the real line. Eventually they leave the unit circle and the associated method becomes
unstable.  Note how CM has higher stability than NAG, and RGD has even higher stability than CM.
}
\end{figure}

In Fig.~\ref{circles}, the shaded gray area represents the unit circle. Any eigenvalue that
leaves this area makes  the associated algorithm unstable. Here we fix $m=\lambda = \gamma = 1$ (other choices are equivalent) and we vary the step size $h > 0$.  These eigenvalues are in general complex and  lie on a circle
which is determined by the amount of friction in the system. Note how for CM and RGD this circle is centered at
the origin, with radius $\sqrt{\mu} \equiv e^{-\gamma h /2}$, since these methods are conformal symplectic and exactly preserve the dissipation of the underlying continuous-time system. However, NAG introduces a spurious damping which is reflected as the circle being translated from the
center, at a distance $1/(e^{\gamma h} + 1)$, and moreover this circle has a smaller radius of $1/(e^{\gamma h} + 1)$ compared to CM and RGD; since this radius is smaller, NAG may have a faster convergence
when these eigenvalues are complex.   As we increase $h$ (left to right in Fig.~\ref{circles}), the eigenvalues move counterclockwise on the circles
until falling on the real line, where one of them goes to the left while the other goes to the right.
Eventually, the leftmost eigenvalue leaves the unit circle for a large enough $h$ (third panel in Fig.~\ref{circles}).
Note that NAG becomes unstable first, followed by CM, and only then by RGD. The main point is that CM and RGD can still be stable for much larger step sizes compared to NAG, and
RGD is even more stable than CM as seen in the rightmost plot in Fig.~\ref{circles}; this is a consequence
of RGD being an integrator of order $r=2$ whereas CM is of order $r=1$. 
Hence, even though NAG may have a slightly faster convergence (due to a stronger contraction), it requires
a smaller step sizes and its stability is more sensitive compared to a conformal symplectic method.
On the other hand, both CM and RGD can operate with larger step sizes, which in practice may even result in a faster solver compared to NAG.

To provide a more quantitative statement, after computing the eigenvalues of the
above transition matrices
for given $\mu \equiv e^{-\gamma h}$, $m$, and $\lambda$, we find the following threshold for stability:
\begin{align}
h_{\textnormal{CM}} &\le \sqrt{ m (1 + \mu + \mu^2 + \mu^3)} \big/ (\mu \sqrt{\lambda}), \\
h_{\textnormal{NAG}} &\le \sqrt{ m(1 + \mu + \mu^2 + \mu^3) } \big/ \sqrt{ \mu \lambda (1 + \mu + \mu^2) }, \\
h_{\textnormal{RGD}} &\le \sqrt{ 2m (1 + \mu + \mu^2 + \mu^3) } \big/ \sqrt{ \mu \lambda (1 + \mu) }.
\end{align}
We can clearly see that RGD has the largest region for $h$, followed by CM, then by NAG, in agreement with the results of Fig.~\ref{circles}.

\section{Numerical experiments}
\label{numerical}

Let us compare RGD (Algorithm~\ref{rgd1}) against NAG \eqref{nag2} and CM \eqref{momgd2} on some test problems.
We stress that all  hyperparameters of each of these methods were systematically  optimized through Bayesian optimization
\cite{Hyperopt} (the default implementation uses a Tree of Parzen estimators).
This yields \emph{optimal} and \emph{unbiased} parameters automatically.
Moreover, by checking the distribution of these hyperparameters during the tuning process  we can get intuition on the sensitivity of each method.
Thus, for each algorithm, we show its convergence rate
in Fig.~\ref{rates} when the best hyperparameters were used.
In addition, in Fig.~\ref{hist} we show the distribution of  hyperparameters during the Bayesian optimization  step---the parameters are indicated and color lines follow Fig.~\ref{rates}.  Such values
are obtained only when the respective algorithm was able to converge. We note that usually CM and NAG diverged more often than RGD which seemed  more robust to parameter choice.
Below we describe some of the optimization problems where such algorithms were tested over.
In Appendix~\ref{sec:extra_numerical} we provide several additional experiments illustrating the benefits of RGD.  The actual code related to our implementation
is extremely simple and can be found at \cite{code}.

\subsection{Correlated quadratic}
Consider $f(x) = (1/2) x^T Q x$ where $Q_{ij} = \rho^{|i-j|}$, $\rho=0.95$,
and $Q$ has size $50 \times 50$---this function was also used in \cite{Betancourt:2018}.
We initialize the position at random, $x_{0,i} \sim \mathcal{N}(0, 10)$, and the velocity as $v_0 = 0$.
The convergence results are shown in Fig.~\ref{corr_quad_rate}.
The distribution of parameters during tuning are  in Fig.~\ref{corr_quad_hist}, showing
that $\alpha \to  1$ is preferable. This gives evidence for an advantage in being conformal symplectic.  Note also that $\delta > 0$, thus ``relativistic effects'' played a
role in improving convergence.

\subsection{Random quadratic}
Consider $f(q) = (1/2) x^T Q  x$  where
$Q$ is a $500\times 500$ positive definite random matrix with
eigenvalues uniformly distributed in $[10^{-3},10]$.
Convergence rates are in Fig.~\ref{random_quad_rate} with the histograms
of parameter search in Fig.~\ref{random_quad_hist}. Again, there is a preference
towards $\alpha \to 1$, evidencing benefits in being conformal symplectic.

\begin{figure}[!t]
\centering
\subfloat[\label{corr_quad_rate} Correlated quadratic]{\includegraphics[width=0.25\textwidth,trim={0 0 0 0},clip]{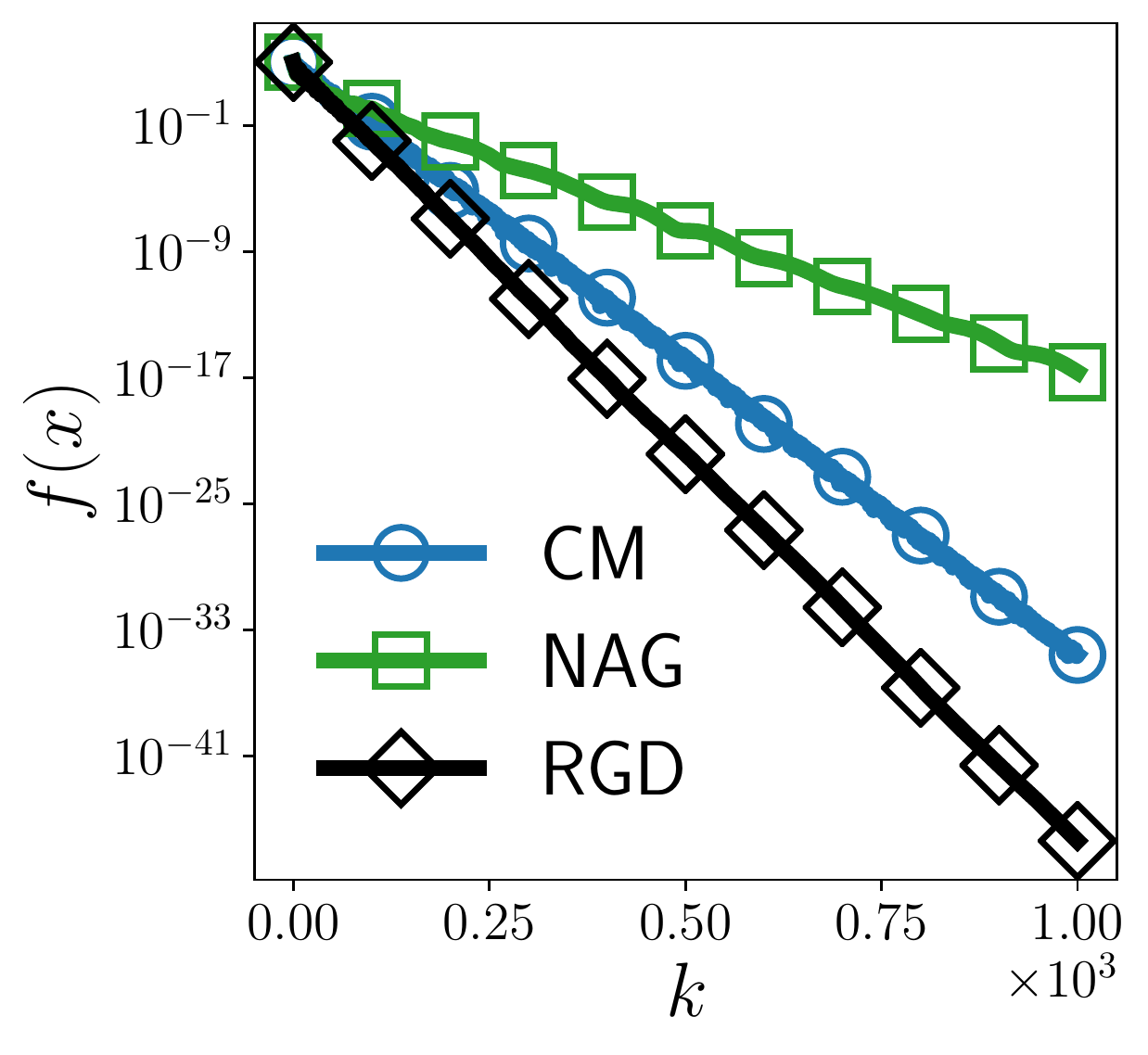}}
\subfloat[\label{random_quad_rate} Random quadratic]{\includegraphics[width=0.25\textwidth,trim={0 0 0 0},clip]{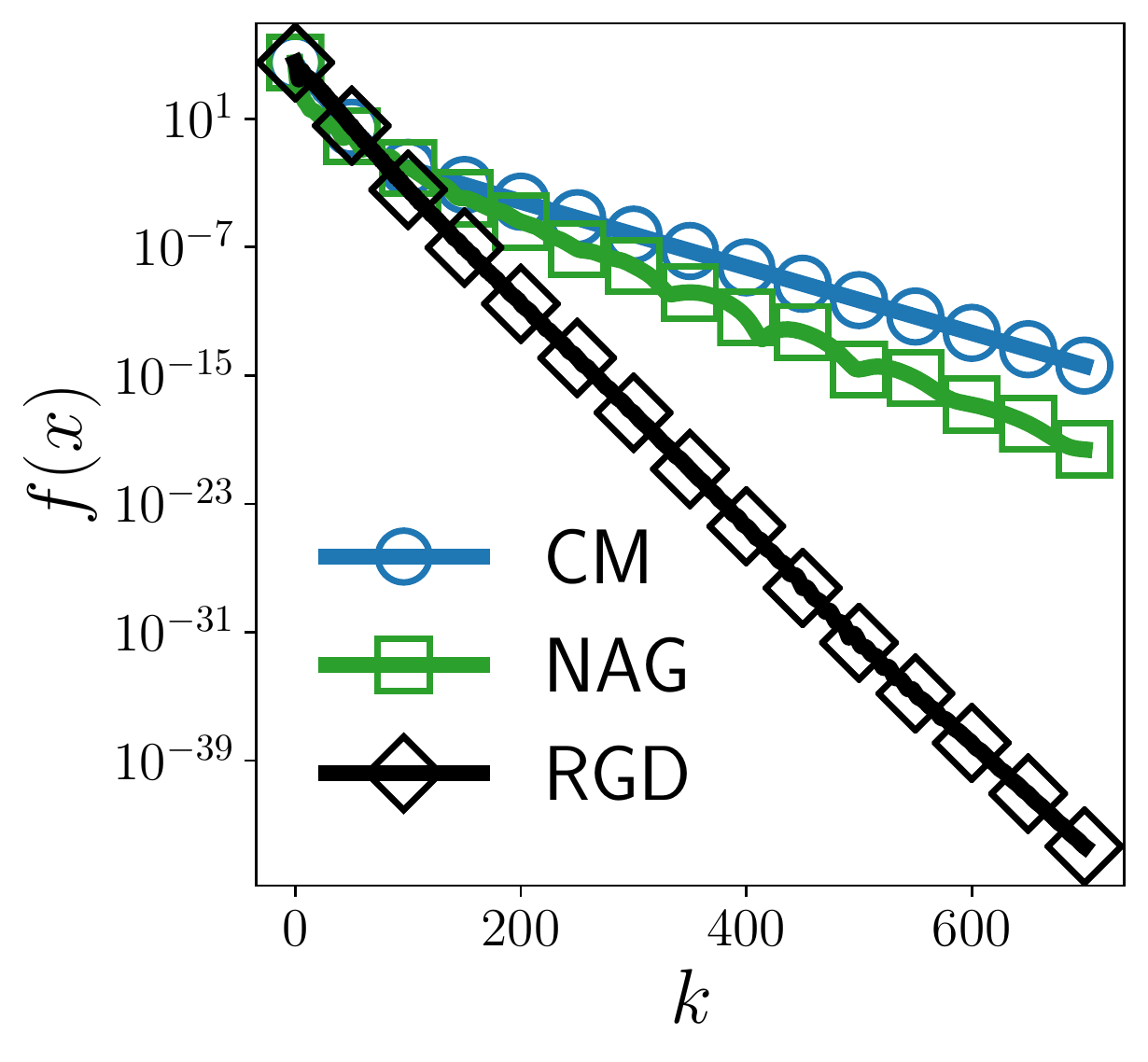}}
\subfloat[\label{ros_rate} Rosenbrock]{\includegraphics[width=0.25\textwidth,trim={0 0 0 0},clip]{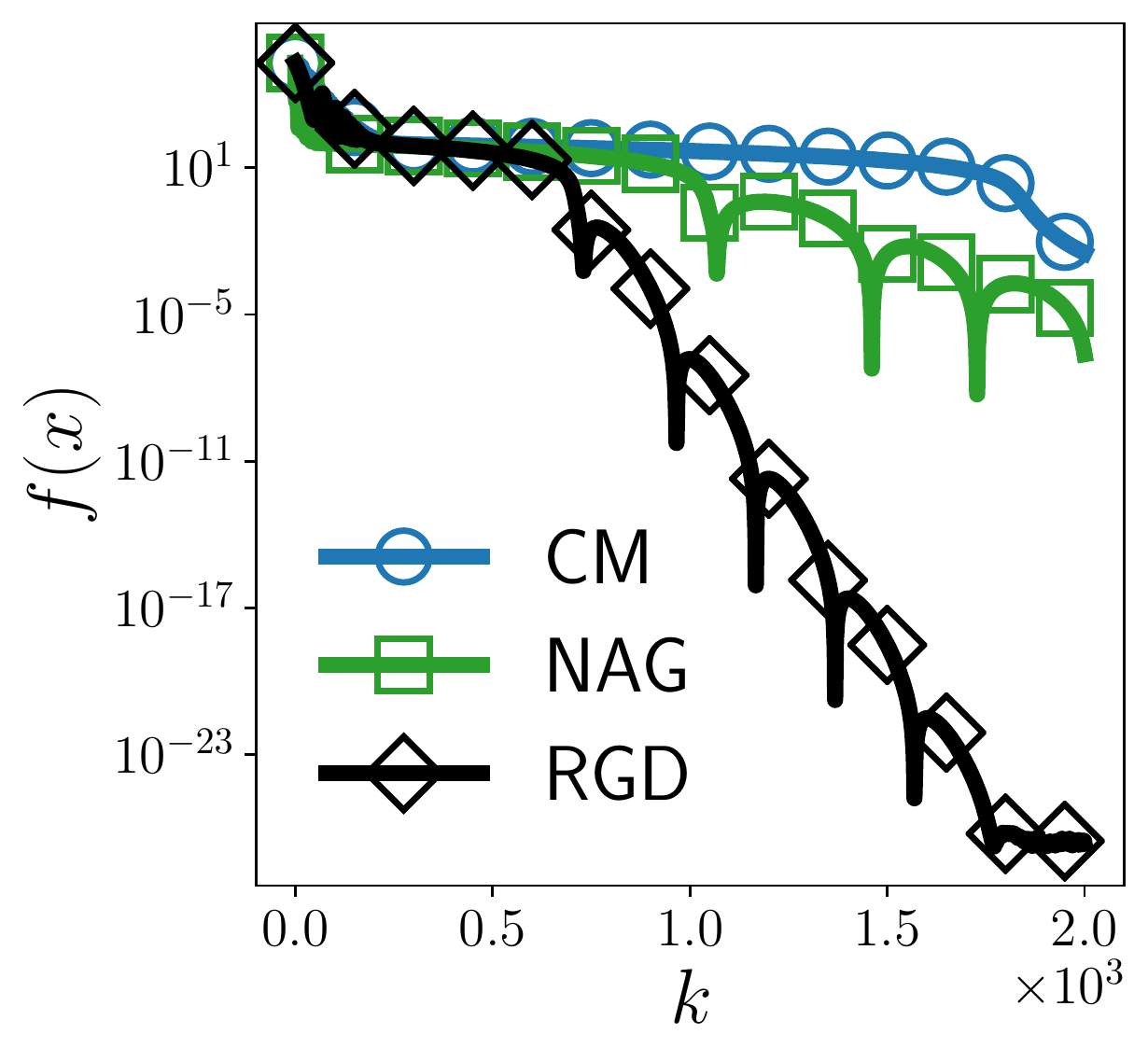}}
\subfloat[\label{mat_comp_rate} Matrix completion]{\includegraphics[width=0.25\textwidth,trim={0 0 0 0},clip]{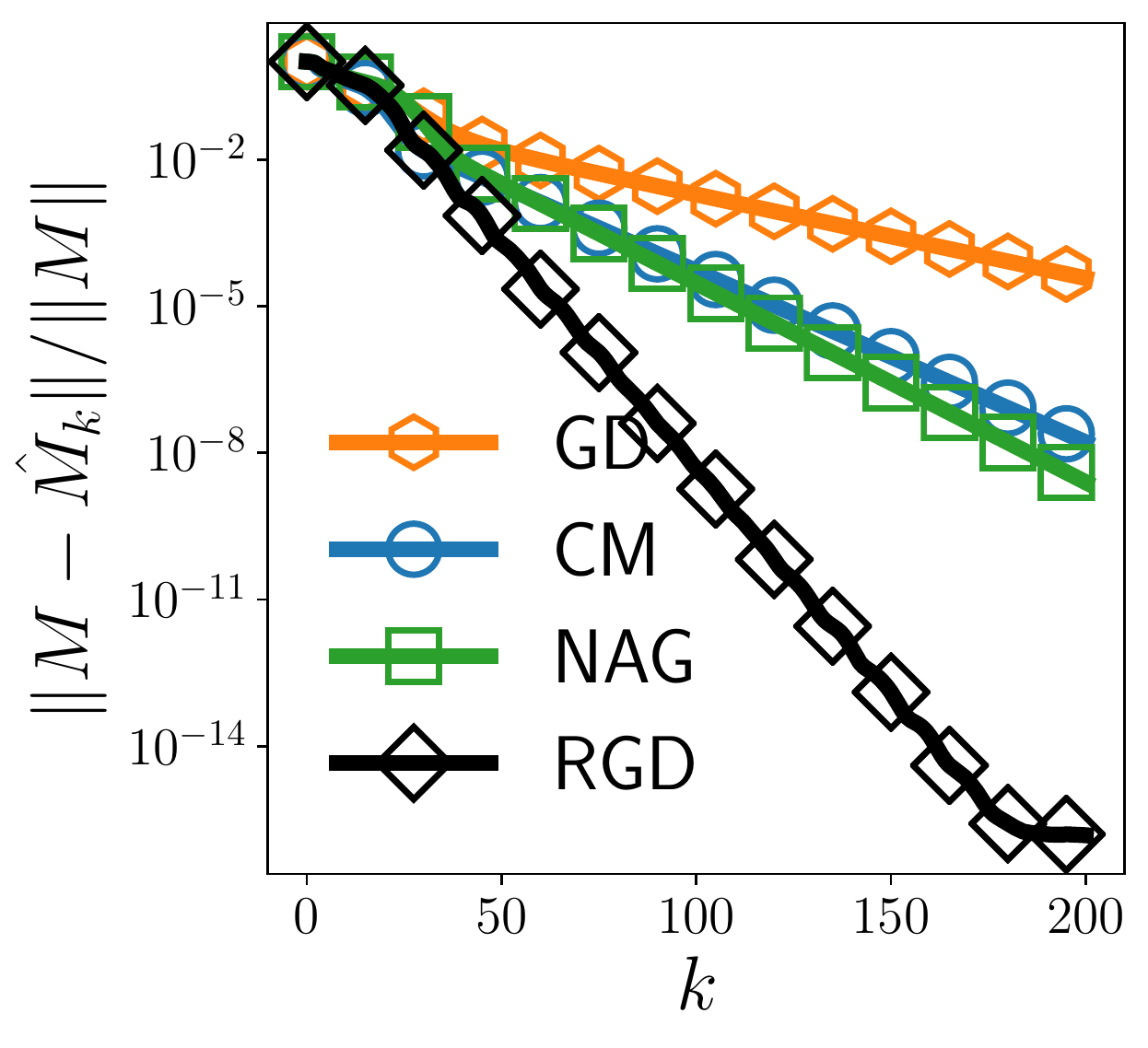}}
\caption{\label{rates}
Convergence rate showing improved performance of RGD (Algorithm~\ref{rgd1}); see text.}
\end{figure}

\begin{figure}[!t]
\centering
\subfloat[\label{corr_quad_hist} Correlated quadratic]{\includegraphics[width=0.25\textwidth,trim={0 0 0 0},clip]{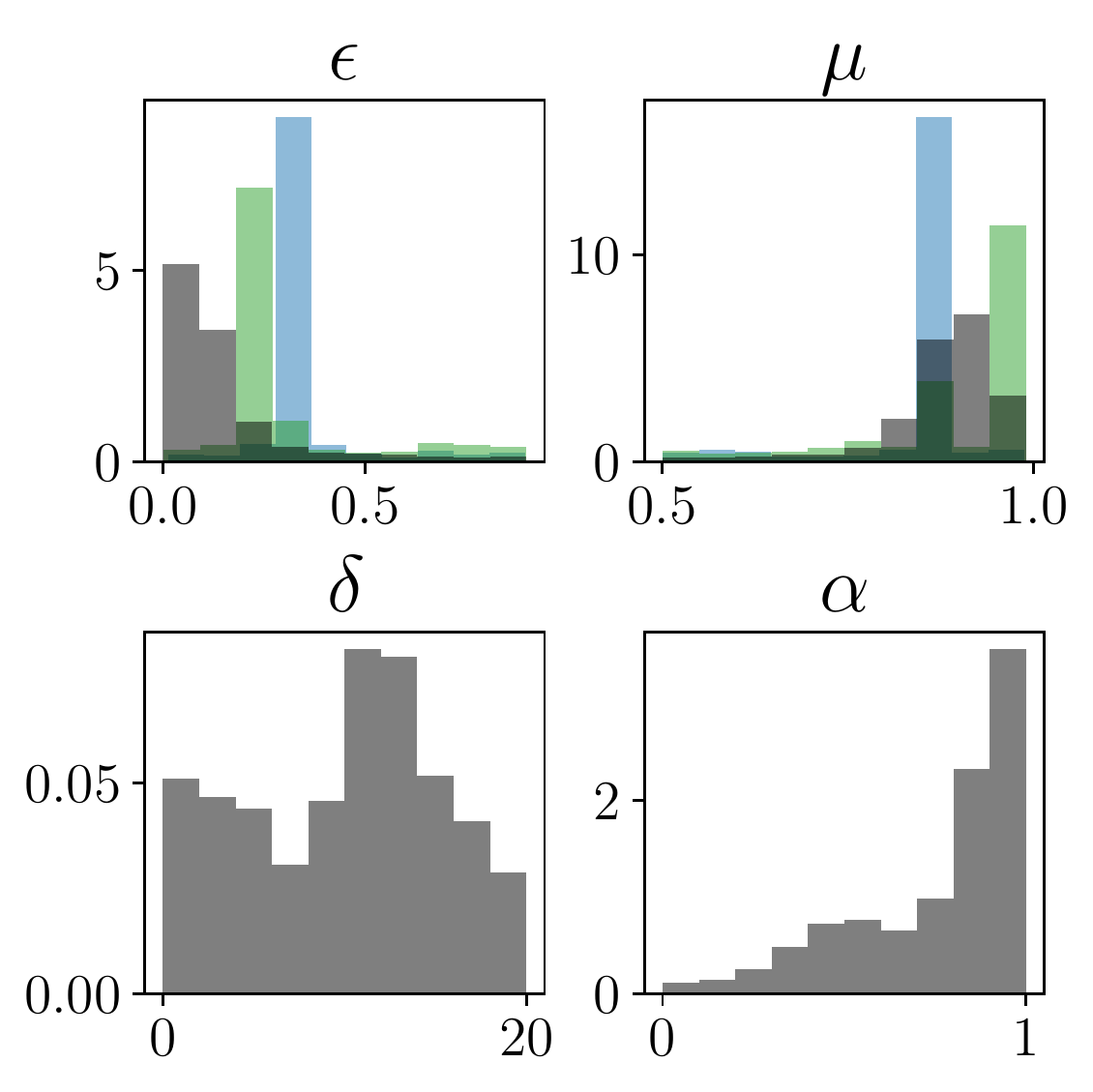}}
\subfloat[\label{random_quad_hist} Random quadratic]{\includegraphics[width=0.25\textwidth,trim={0 0 0 0},clip]{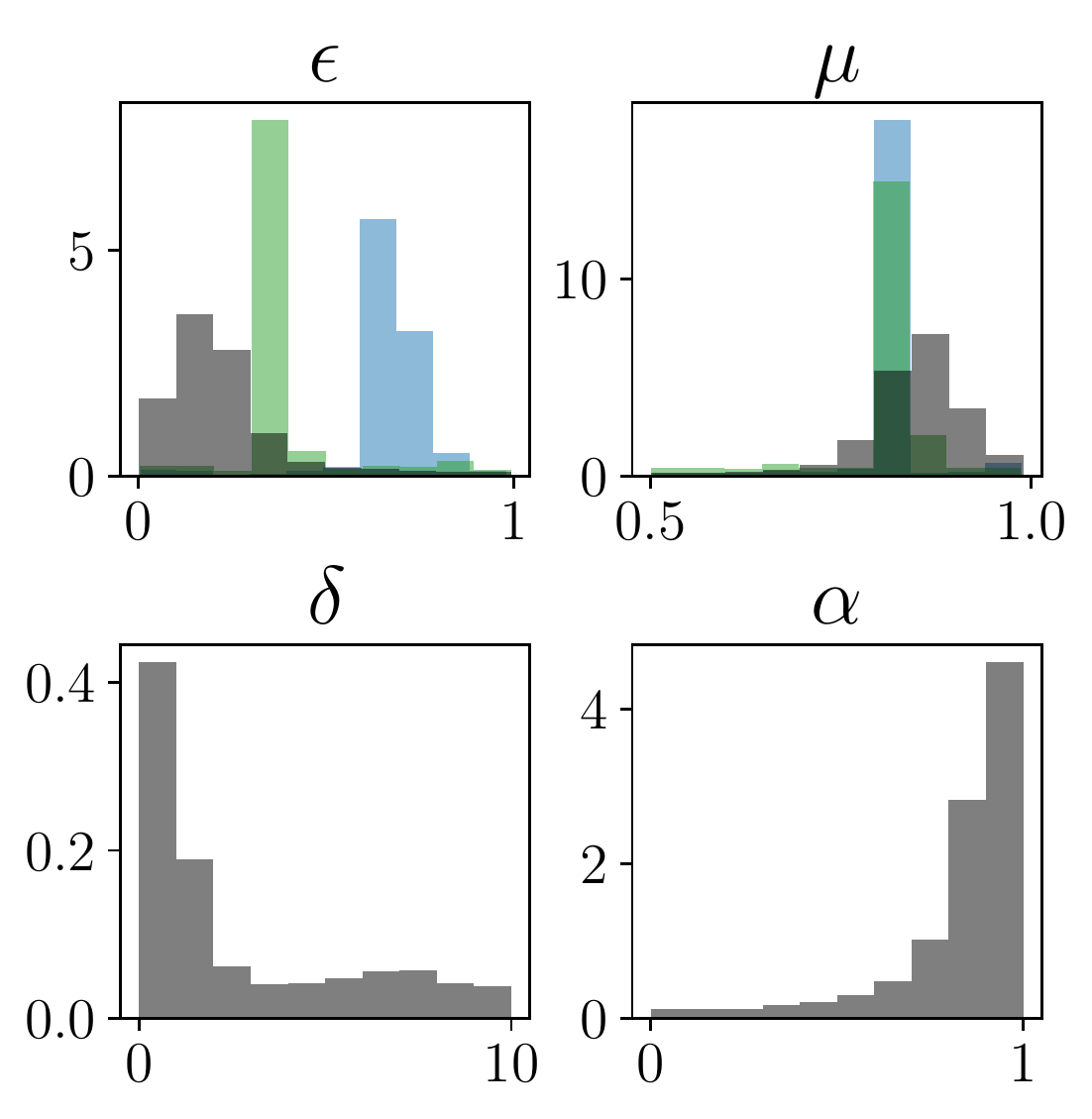}}
\subfloat[\label{ros_hist} Rosenbrock]{\includegraphics[width=0.25\textwidth,trim={0 0 0 0},clip]{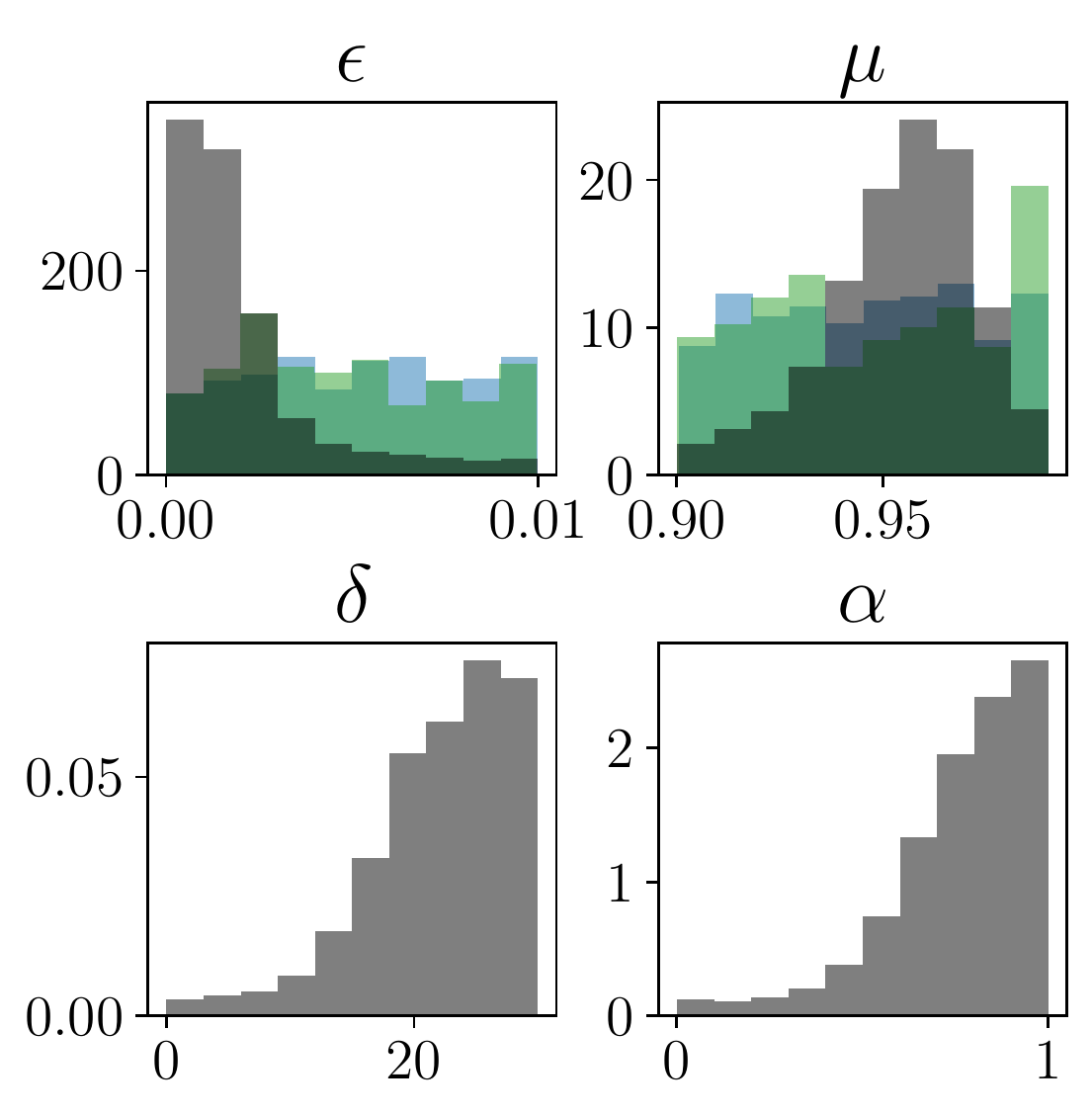}}
\subfloat[\label{mat_comp_hist} Matrix completion]{\includegraphics[width=0.25\textwidth,trim={0 0 0 0},clip]{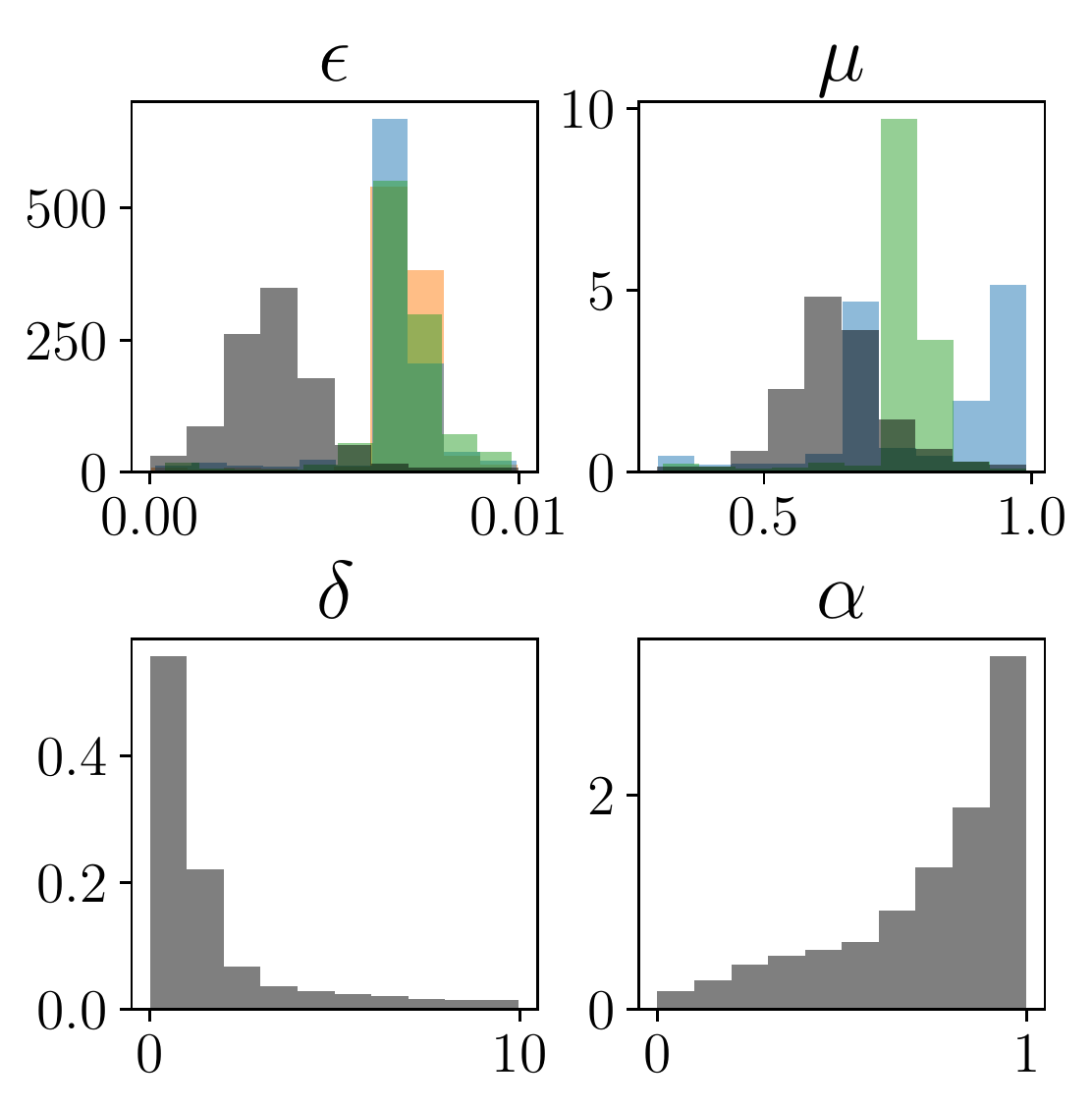}}
\caption{\label{hist}
Histograms of hyperparameter tuning by Bayesian optimization. 
Tendency towards $\alpha \approx 1$ indicates benefits of being symplectic, while $\alpha \approx 0$
of being extra damped as in NAG. Tendency towards $\delta > 0$ indicates benefits
of relativistic normalization. (Colors follow Fig.~\ref{rates}.)}
\end{figure}

\subsection{Rosenbrock}
For a challenging problem in higher dimensions,
consider the nonconvex Rosenbrock function
$f(x) \equiv \sum_{i=1}^{n-1}\big(100(x_{i+1}-x_i^2)^2
  +(1-x_i)^2\big)$ with
$n=100$
\cite{Rosenbrock:1960,Goldberg:1989};
this case
was already studied in detail \cite{Kok:2009}.
Its landscape is quite involved, e.g.
there are two minimizers,
one global at $x^\star = (1,\dotsc,1)^T$ with $f(x^\star) = 0$ and
one local near $x\approx(-1,1,\dotsc,1)^T$ with
$f \approx 3.99$. There are also---exponentially---many
saddle points \cite{Kok:2009}, however
only two of these are actually hard to escape.
These four stationary points account for $99.9\%$ of the
solutions found by Newton's method \cite{Kok:2009}.
We note that both minimizers lie on a flat, deep, and narrow valley, making
optimization challenging.
In Fig.~\ref{ros_rate} we have the convergence of each method
initialized at $x_{0,i}=\pm 2$ for $i$ odd/even.
Fig.~\ref{ros_hist}  shows histograms for parameter selection.
Again, we see the favorable symplectic tendency, $\alpha\to 1$.
Here relativistic effects, $\delta \ne 0$, played a predominant role in
the improved convergence of RGD.

\subsection{Matrix completion}
Consider an $n\times n$ matrix $M$ of rank $r \ll n$ with observed entries
in the support $(i,j) \in \Omega$, where
$P_{\Omega}(M)_{ij} = M_{ij}$ if $(i,j) \in \Omega$ and
$P_{\Omega}(M)_{ij} = 0 $ projects onto this support.
The goal is to recover $M$ from the knowledge of $P_{\Omega}(M)$.
We assume that the rank $r$ is known.  In this case, if the number of observed
entries is $O(rn)$ it is possible to recover $M$ with high
probability \cite{Montanari:2009}.
We do this by solving the nonconvex problem
$\min_{U,V} \|  P_{\Omega}(M - UV^T) \|_F^2$,
where $U, V \in \mathbb{R}^{n\times r}$, by \emph{alternating minimization}:
for each iteration we apply the previous algorithms first on $U$ with $V$ held
fixed, followed by similar updates for $V$ with the new $U$ fixed.
This is a know technique for gradient descent (GD), which we additionally include as a baseline.
We generate $M = RS^T$ where $R,S \in \mathbb{R}{n\times r}$ have iid entries
from the normal distribution $\mathcal{N}(1,2)$.
We initialize $U$ and $V$ sampled from the standard normal.
The support is chosen uniformly at random
with sampling ratio $s = 0.3$, yielding $p = s n^2$ observed entries.
We set $n=100$ and $r = 5$.
This gives a number of effective degrees of freedom $d = r(2n - r)$ and the ``hardness''
of the problem can be quantified via $d/p \approx 0.325$.
Fig.~\ref{mat_comp_rate} shows the convergence rate, and
Fig.~\ref{mat_comp_hist} the parameter search.

\section{Discussion and outlook}
\label{final}

This paper introduces a new perspective on a recent line of research connecting accelerated optimization
methods to continuous-time dynamical systems that have been playing a major role in  machine learning.
We brought \emph{conformal symplectic} techniques for \emph{dissipative systems} into this
context, besides proposing a new method called \emph{Relativistic Gradient Descent} (RGD), based on a dissipative relativistic system; see Algorithm~\ref{rgd1}. RGD generalizes both the classical momentum (CM) or heavy ball method---given by \eqref{momgd2}---as well as Nesterov's accelerated gradient (NAG)---given by \eqref{nag2}; each of these methods are recovered as particular cases from RGD which has no additional computational cost compared to CM and NAG. Moreover, RGD has more flexibility, can interpolate between a conformal symplectic behaviour or introduce some  Hessian dependent damping in the spirit of  NAG,  and has potential to control instabilities due to large gradients by normalizing the momentum. In our experiments, RGD significantly
outperformed CM and NAG, specially in settings with large gradients or functions with a fast growth; besides Section~\ref{numerical} we report several additional examples in Appendix~\ref{sec:extra_numerical}.

We also elucidated what is the symplectic structure behind CM and NAG.
We found that
the former turns out to be a conformal symplectic integrator (Corollary~\ref{momgd_symp}),
thus being ``dissipative-preserving,'' while the latter introduces a spurious contraction of the symplectic form by a Hessian driven damping (Theorem~\ref{nag_contract_omega}).
This is an effect of second order in the step size but  may affect convergence and stability.
We pointed out a tradeoff between this extra contraction and
the stability of a conformal symplectic method.
We
also derived modified or perturbed equations for CM and NAG, describing these methods
to a higher degree of resolution; this analysis provides several new insights into these methods  and may form the basis for exploring these algorithms using different techniques compared to standard approaches in pure optimization.

On a higher level, this paper shows  how structure-preserving
discretizations of classical dissipative systems can be useful for studying existing optimization algorithms, as well as introduce new methods inspired by real  physical systems.
A thorough justification for the use of structure-preserving---or ``dissipative symplectic''---discretizations in this context was recently provided in \cite{Franca:2020} under great generality.

Finally, a more refined analysis of RGD is certainly an interesting future problem, though considerably
challenging due to the nonlinearity introduced by the $\sqrt{1+ \delta \| v\|^2}$ term in the updates
of Algorithm~\ref{rgd1}.
To give an example, even if one assumes a simple quadratic function $f(x) = (\lambda/2)  x^2$, the
differential equation \eqref{confrel} is nonlinear and does not admit a closed form solution, contrary to the differential
equation associated to CM and NAG which is linear and can be readily integrated. Thus, even
in continuous-time, the analysis for RGD is likely to be involved.
Finally, it would be interesting to consider RGD in a stochastic setting, namely investigate its diffusive properties in a random media, which may bring benefits
to nonconvex optimization and sampling.

\bigskip

\subsection*{Acknowledgments}
GF would like to thank Michael I. Jordan for discussions.
This work was supported by grants ARO MURI W911NF-17-1-0304,  
NSF 2031985, and NSF 1934931.

\bigskip

\appendix


\section{Order of accuracy of the general integrators} \label{proof_accuracy}

It is known that a composition of the type $\Psi_h^{A}\circ \Psi_h^{B}$, where $A$ and $B$ represents the components
of distinct vector fields, leads
to an integrator of order $r=1$, whereas a composition in the form
$\Psi_{h/2}^A \circ \Psi_h ^B \circ \Psi_{h/2}^A$ leads to an integrator of order $r=2$ \cite{Hairer}---the latter
is known as Strang splitting. However, here we provide an explicit and direct proof of these facts
for the generic integrators \eqref{gen_euler} and \eqref{gen_leap}, respectively.

\begin{proof}[Proof of Theorem~\ref{accuracy}]
From the equations of motion \eqref{confhameq} and Taylor expansions:
\begin{equation}
\label{taylorq}
\begin{split}
x(t_k + h) &= x + h \dot{x} + \dfrac{h^2}{2} \ddot{x} + O(h^3) \\
&= x + h \nabla_p H + \dfrac{h^2}{2}\left( \nabla^2_{xp} H \dot{x} + \nabla^2_{pp} H \dot{p}\right) + O(h^3) \\
&= x + h \nabla_p H + \dfrac{h^2}{2} \nabla^2_{xp} H \nabla_p H - \dfrac{h^2}{2} \nabla^2_{xp}\nabla_x H - \dfrac{h^2}{2} \gamma \nabla^2_{pp}H p + O(h^3),
\end{split}
\end{equation}
and
\begin{equation}
\label{taylorp}
\begin{split}
p(t_k + h) &= p + h \dot{p} + \dfrac{h^2}{2} \ddot{p} + O(h^3) \\
&= p - h \nabla_x H - h \gamma p + \dfrac{h^2}{2}\left( - \nabla^2_{xx} H \dot{x} - \nabla^2_{xp} H \dot{p} - \gamma \dot{p} \right) + O(h^3) \\
&= p - h \nabla_x H - h \gamma p -\dfrac{h^2}{2} \nabla^2_{xx}H \nabla_p H + \dfrac{h^2}{2}\nabla^2_{xp} H \nabla_x H + \dfrac{h^2}{2}\gamma \nabla^2_{xx}H p \\
& \qquad + \dfrac{h^2}{2} \gamma \nabla_x H + \dfrac{h^2}{2} \gamma^2 p + O(h^3),
\end{split}
\end{equation}
where we denote $x \equiv x(t_k)$ and $p\equiv p(t_k)$ for $t_k = k h$ ($k=0,1,\dotsc$),
and it is implicit that all gradients and Hessians of $H$ are being computed at $(x,p)$.

Consider \eqref{gen_euler}.
Under one step of this map, starting from
the point $(x,p)$, upon using Taylor expansions
we have
\begin{equation}
\label{expan1}
x_{k+1} = x + h\nabla_p H + O(h^2)
\end{equation}
and
\begin{equation}
\label{expan2}
p_{k+1} = e^{-\gamma h} p - h \nabla_x H + O(h^2)
= p - \gamma h p -h \nabla_x H(x, p) + O(h^2).
\end{equation}
Comparing these last two equations with \eqref{taylorq} and \eqref{taylorp} we conclude that
\begin{equation}
  x_{k+1} = x(t_{k} + h) + O(h^2), \qquad
  p_{k+1} = p(t_{k} + h) + O(h^2).
\end{equation}
Therefore, the discrete state approximates the continuum state up
to an error of $O(h^2)$,
 obeying Definition~\ref{order} with $r=1$.

The same approach is applicable to the numerical map \eqref{gen_leap}.
Expanding the first update:
\begin{equation}
\label{Xtilde_exp}
\begin{split}
\tilde{X} &= x + \dfrac{h}{2} \nabla_p H\big(x + (h/2) \nabla_p H, p - (h/2)\gamma p \big) + O(h^3), \\
&= x + \dfrac{h}{2} \nabla_p H + \dfrac{h^2}{4} \nabla^2_{xp}H \nabla_p H - \dfrac{h^2}{4} \gamma \nabla^2_{pp}H p + O(h^3).
\end{split}
\end{equation}
Expanding the second update:
\begin{equation}
\label{Ptilde_exp}
\begin{split}
\tilde{P} &= e^{-\gamma h /2} p - \dfrac{h}{2}  \nabla_x H\big(x + (h/2) \nabla_p H, p - (h/2)\gamma p \big) \\
&\qquad - \dfrac{h}{2} \nabla_xH\big(x + (h/2) \nabla_p H, p - (h/2)\gamma p
- h \nabla_x H\big)  + O(h^3), \\
&= e^{-\gamma h/2} p  -h \nabla_x H - \dfrac{h^2}{2}\nabla^2_{xx}H \nabla_p H + \dfrac{h^2}{2}\gamma \nabla^2_{xp} H p + \dfrac{h^2}{2} \nabla^2_{xp} H \nabla_x H + O(h^3).
\end{split}
\end{equation}
Making use of \eqref{Xtilde_exp} and \eqref{Ptilde_exp} we thus find:
\begin{equation}
\label{X_exp}
\begin{split}
X &= \tilde{X} + \dfrac{h}{2}\nabla_p H(\tilde{X}, \tilde{P}) \\
&= x + \dfrac{h}{2} \nabla_p H + \dfrac{h^2}{4} \nabla^2_{xp} H \nabla_p H - \dfrac{h^2}{4}\gamma \nabla^2_{pp} H p \\
& \qquad + \dfrac{h}{2} \nabla H\big(x + (h/2)\nabla_p H, p - (h/2)\gamma p - h \nabla_x H\big) + O(h^3) \\
&= x + h \nabla_p H + \dfrac{h^2}{2}\nabla^2_{xp} H \nabla_p H - \dfrac{h^2}{2}\gamma \nabla_{pp}^2 H p - \dfrac{h^2}{2}\nabla_{pp}^2 H \nabla_x H + O(h^3).
\end{split}
\end{equation}
Comparing with \eqref{taylorq} we conclude that
\begin{equation}
x_{k+1}  = x(t_k + h) + O(h^3).
\end{equation}
Finally, from \eqref{Ptilde_exp} we have
\begin{equation}
\begin{split}
P &= e^{-\gamma h/2} \tilde{P} \\
&= e^{\gamma h} p - e^{-\gamma h/2}\Big\{
 h \nabla_x H + \dfrac{h^2}{2} \nabla^2_{xx} H \nabla_p H + \dfrac{h^2}{2} \gamma \nabla_{xp}^2 H p - \dfrac{h^2}{2} \nabla^2_{xp} H \nabla_x H
\Big\} + O(h^3)\\
&= p -\gamma h p + \dfrac{h^2}{2} \gamma^2 p - h \nabla_x H - \dfrac{h^2}{2} \nabla^2_{xx} H \nabla_p H + \dfrac{h^2}{2}\gamma \nabla_{xp}^2 H p \\ &\qquad + \dfrac{h^2}{2}\nabla^2_{xp} H \nabla_{x}H + \dfrac{h^2}{2} \gamma \nabla_x H + O(h^3).
\end{split}
\end{equation}
Comparing this with \eqref{taylorp} implies
\begin{equation}
p_{k+1} = p(t_k + h) + O(h^3).
\end{equation}
Therefore, in this case we satisfy Definition~\ref{order}
with $r=2$.
\end{proof}

From the above general results it is immediate that:
\begin{itemize}
\item CM \eqref{momgd2}---or equivalently \eqref{simpleap} which is more appropriate to make connections with the  continuous-time system---is a first order integrator
to the conformal Hamiltonian system \eqref{confhameq} with the classical
 Hamiltonian \eqref{hamiltonian1}; the equations of motion are explicitly given by \eqref{class_harm}.
\item The relativistic extension of CM given by \eqref{rgd0} is a first order
 integrator to the conformal relativistic Hamiltonian system \eqref{confrel}.
 \item  RGD \eqref{rgd} with $\alpha=1$---also equivalently written as Algorithm.~\ref{rgd1}---is a second order integrator to system \eqref{confrel}.
\end{itemize}

\section{Additional numerical experiments} \label{sec:extra_numerical}

Here we  compare RGD (Algorithm~\ref{rgd1}) with CM \eqref{momgd2} and NAG \eqref{nag2}
on several additional test functions;
for details on these functions see e.g. \cite{Jamil:2013} and references therein.  We follow the  procedure already described in Section~\ref{numerical} where we optimized the hyperparameters of these algorithm using Bayesian optimization.\footnote{We provide the actual code used in our numerical simulations in \cite{code}.}
 We report the convergence rate using the best parameters found together with histograms of the parameter search. In all cases we initialize the velocity as $v_0 = 0$.
The initial position $x_0$ was chosen inside the range where the corresponding test function
is usually considered.

First we consider functions with a quadratic growth. These results are shown in Figs.~\ref{fig:booth}--\ref{fig:sumsquares}.  In this case RGD performed similarly to CM and NAG, although with some improvement. In any case RGD proved to be more stable, i.e. it worked well for a wider range of hyperparameters.

We expect that RGD stands out on settings with large gradients or objective functions with  fast growing tails. Therefore, in the remaining figures, i.e.  Fig.~\ref{fig:beale}--\ref{fig:rosenbrock}, we consider  more challenging optimization problems with functions that grow stronger than a quadratic.  For some of these problems the minimum lies on a flat valley, making it hard for an algorithm to stop around the minimum after gaining a lot of speed from a very steep descent direction.  Note that in all these cases the improvement of RGD over CM and NAG is significant, and the parameter $\delta$---which controls  relativistic effects---had an important role.
The conformal symplecticity, which is indicated by the tendency $\alpha \to 1$, also brings an improved stability in the discretization.
These results provide compelling evidence for the benefits of RGD.

\bigskip

\begin{figure}[!h]
\centering
\includegraphics[width=0.33\textwidth,trim={40 0 10 0}]{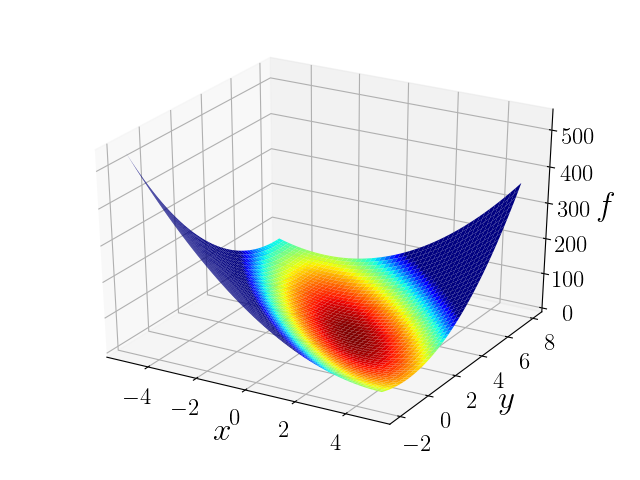}%
\includegraphics[width=0.66\textwidth]{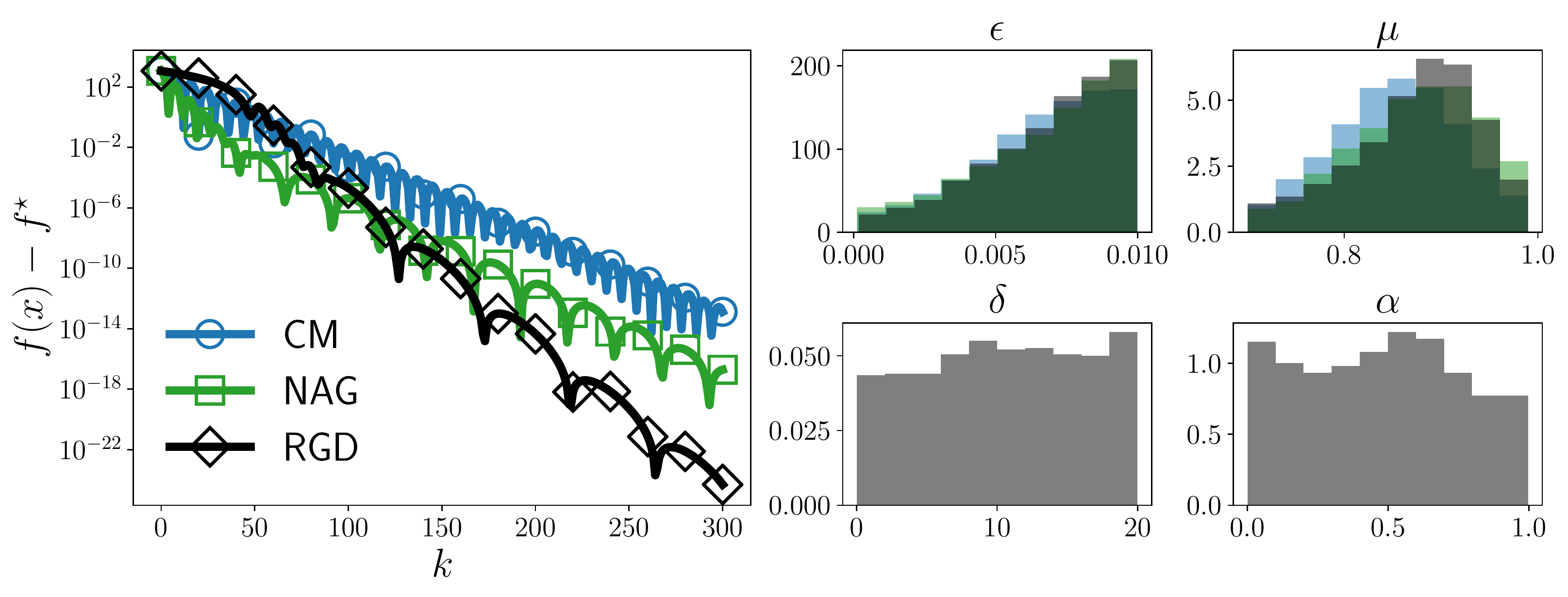}
\caption{\label{fig:booth} \emph{Booth function}: $f(x,y) \equiv (x+2y-7)^2 + (2x+y-5)^2$. Global minimum at $f(1,3) = 0$. We initialize at $x_0 = (10,10)$. This function is usually evaluated on the region $-10 \le x,y \le 10$. All methods perform well on this problem which is not challenging.}
\end{figure}
\begin{figure}[!h]
\centering
\includegraphics[width=0.33\textwidth,trim={10 0 40 0}]{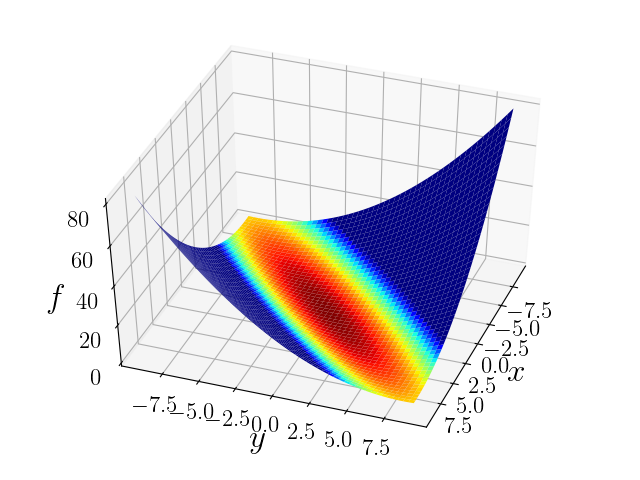}%
\includegraphics[width=0.66\textwidth]{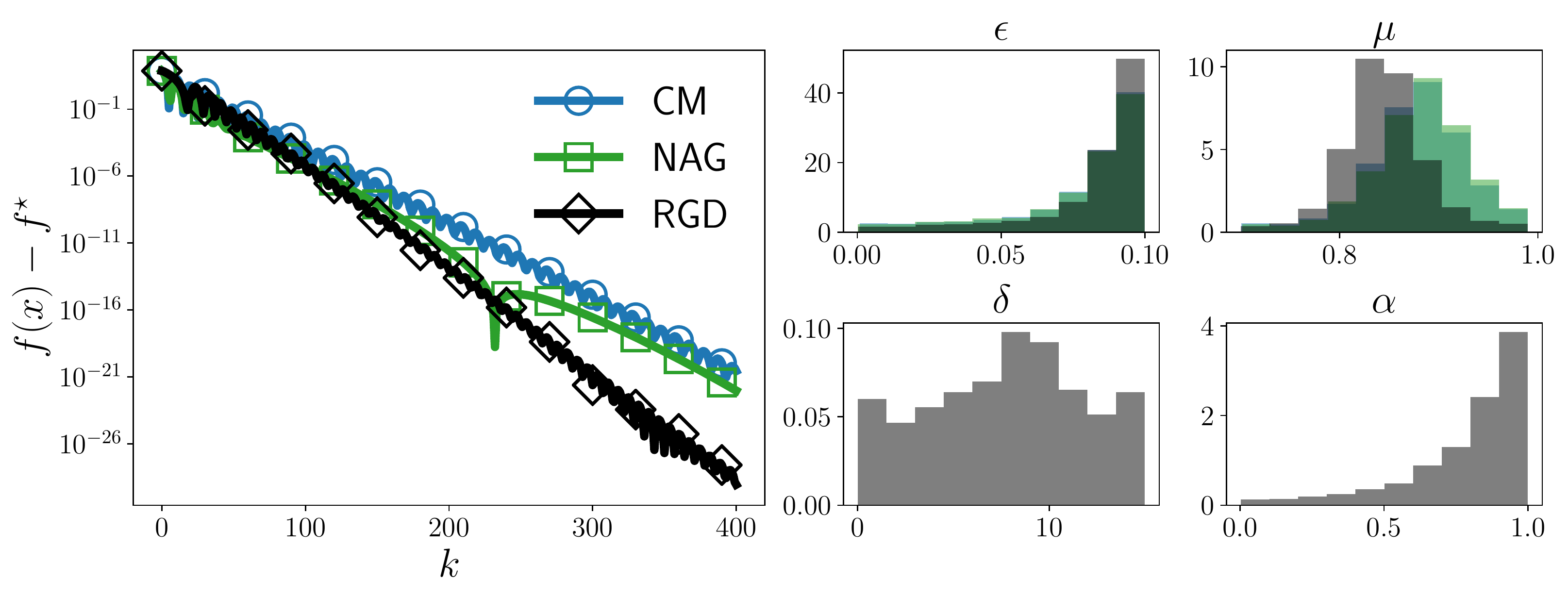}
\caption{\label{fig:matyas} \emph{Matyas function}: $f(x,y) \equiv 0.26(x^2+y^2) - 0.48 x y$. Global minimum is at $f(0,0) = 0$. We initialize at $x_0 = (10,-7)$. This function is usually evaluated on the region $-10 \le x,y \le 10$. Even though the function has a---not so strong---quadratic growth, we see a slight improvement of RGD; note $\delta > 0$. Note also the ``symplectic tendency'' $\alpha \to 1$.}
\end{figure}
\begin{figure}[!h]
\centering
\includegraphics[width=0.33\textwidth,trim={40 0 0 0}]{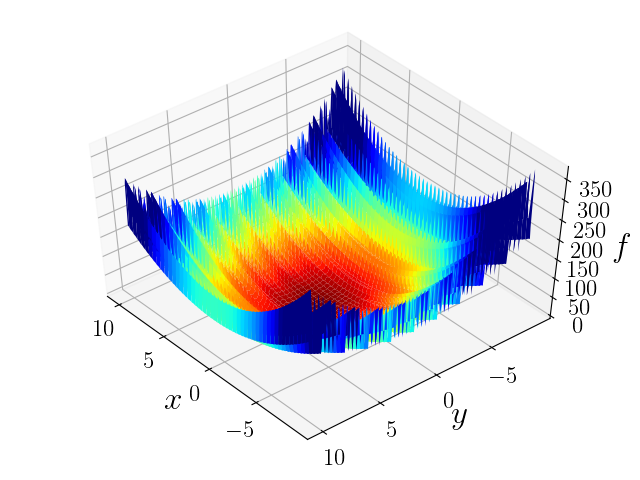}%
\includegraphics[width=0.66\textwidth]{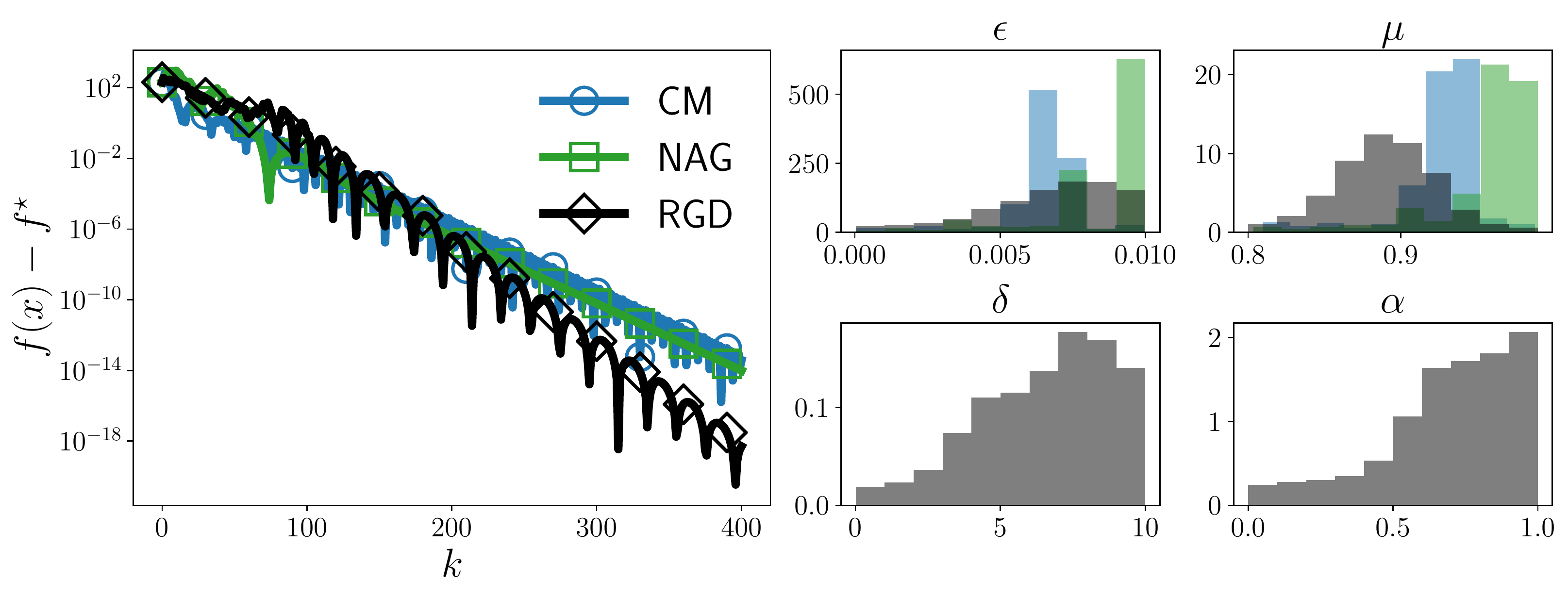}
\caption{\label{fig:levi13} \emph{L\' evi function \#13}: $f(x,y) \equiv \sin^2 3 \pi x + (x-1)^2(1+\sin^2 3\pi y) + (y-1)^2(1+\sin^2 2\pi y)$. It is multimodal with the global minimum at $f(1,1) = 0$. We initialize at $x_0 = (10,-10)$. This function is usually studied on the region $-10 \le x,y \le 10$.  Although this function is nonconvex, the optimization problem is not very challenging. However, we noticed that CM and NAG got stuck on a local minimum more often than RGD when running this example multiple times.
}
\end{figure}
\begin{figure}[!h]
\centering
\includegraphics[width=0.33\textwidth,trim={10 0 40 0}]{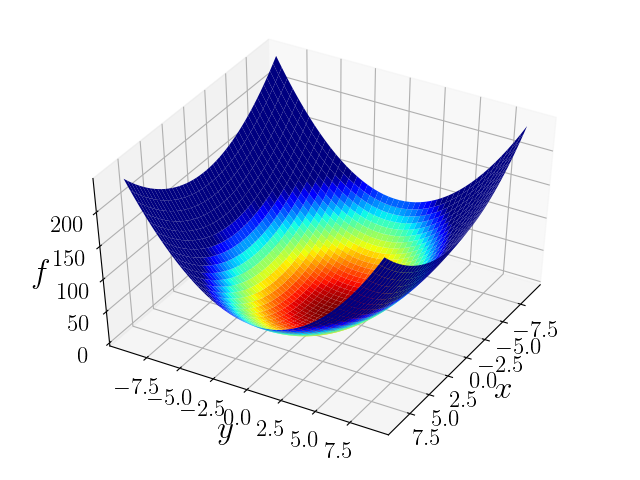}%
\includegraphics[width=0.66\textwidth]{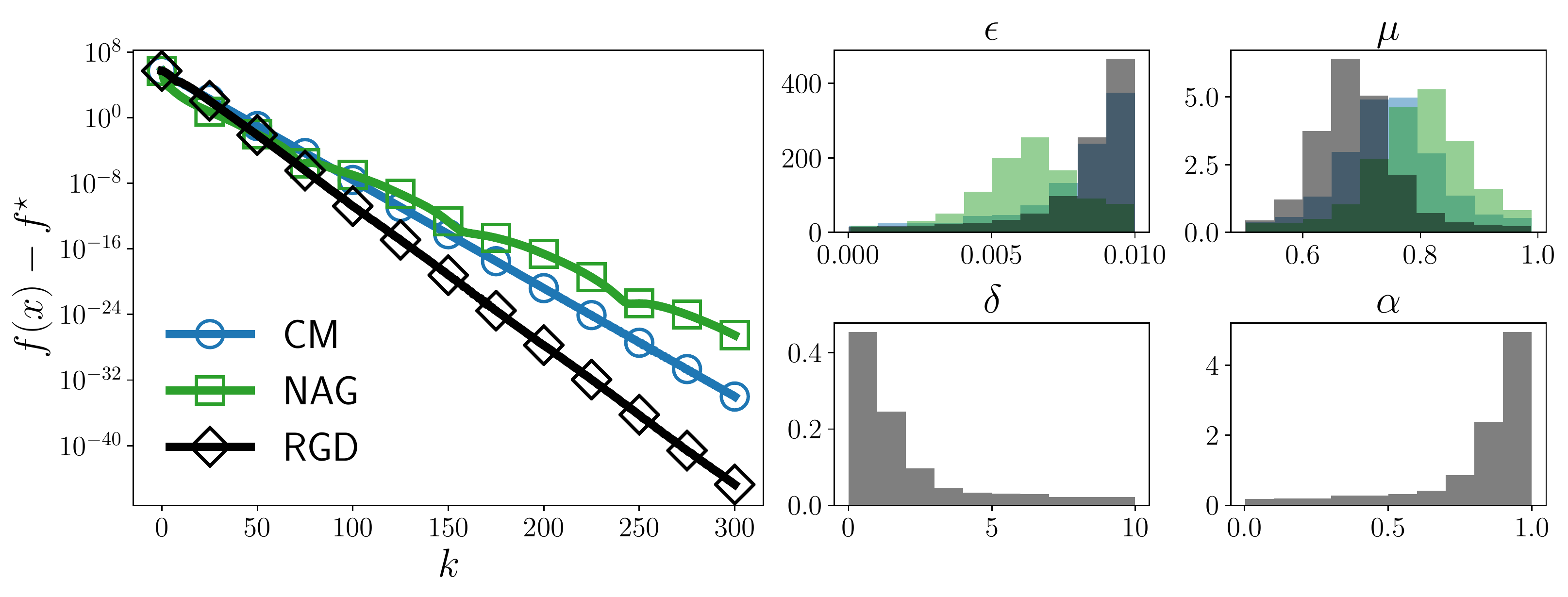}
\caption{\label{fig:sumsquares} \emph{Sum of squares}: $f(x) \equiv \sum_{i=1}^n i x_i^2$. The minimum is at $f(0) = 0$. We consider $n=100$ dimensions and initialize at $x_{0} = (10, \dotsc,10)$. The usual region of study is $-10 \le x_i \le 10$.
Note that there is a clear tendency towards $\alpha \to 1$ in this case, i.e. in being conformal symplectic.}
\end{figure}
%
\begin{figure}[!h]
\centering
\includegraphics[width=0.33\textwidth,trim={10 0 40 0}]{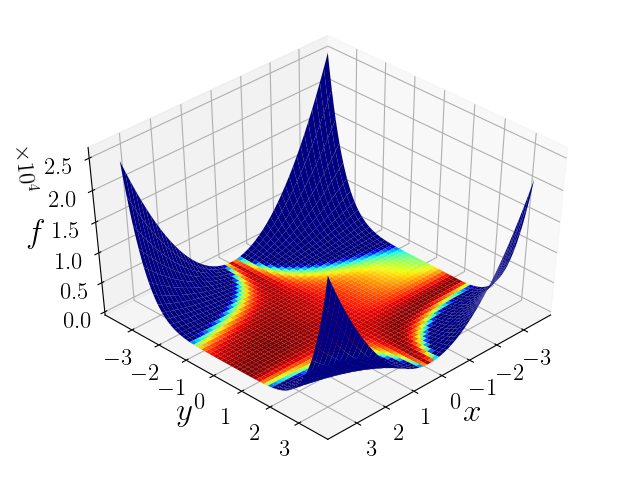}%
\includegraphics[width=0.66\textwidth]{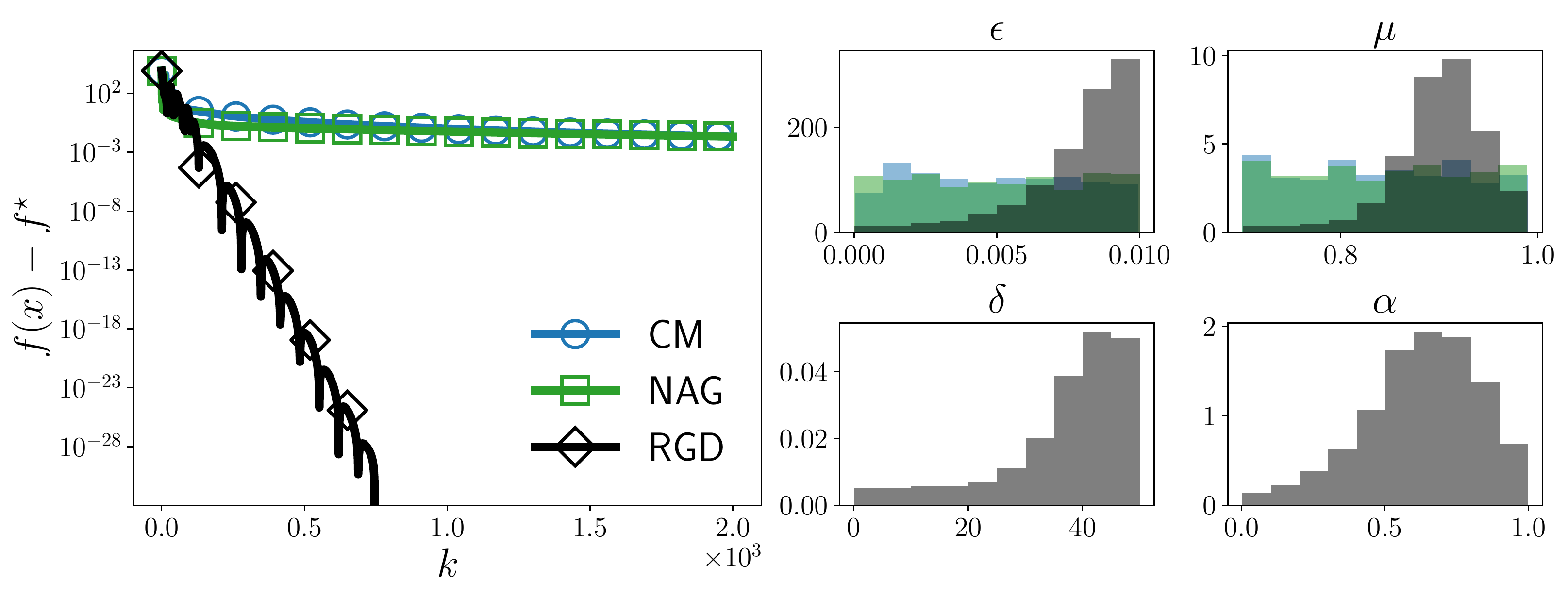}
\caption{\label{fig:beale} \emph{Beale function}: $f(x,y) \equiv (1.5-x + xy)^2 + (2.25 - x + xy^2)^2 + (2.625-x+xy^3)^2$. The global minimum is at $f(3,1/2) = 0$, lying on a flat and narrow valley which makes optimization challenging. Note also that this functions grows stronger than a quadratic.
This function is usually considered on the region $-4.5 \le x,y \le 4.5$.
We initialize at $x_0 = (-3,-3)$.
Note how CM and NAG were unable to minimize the function, while RGD was able to find the global minimum to high accuracy; $\delta \gg 0$ played a predominant role, indicating benefits from ``relativistic effects.''}
\end{figure}
\begin{figure}[!h]
\centering
\includegraphics[width=0.33\textwidth,trim={0 0 40 0}]{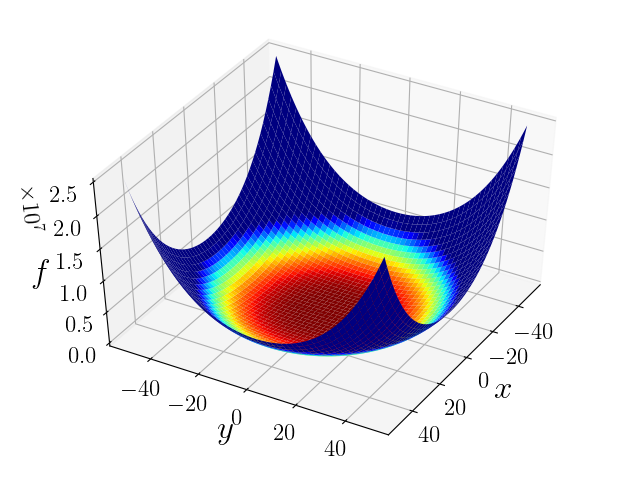}%
\includegraphics[width=0.66\textwidth]{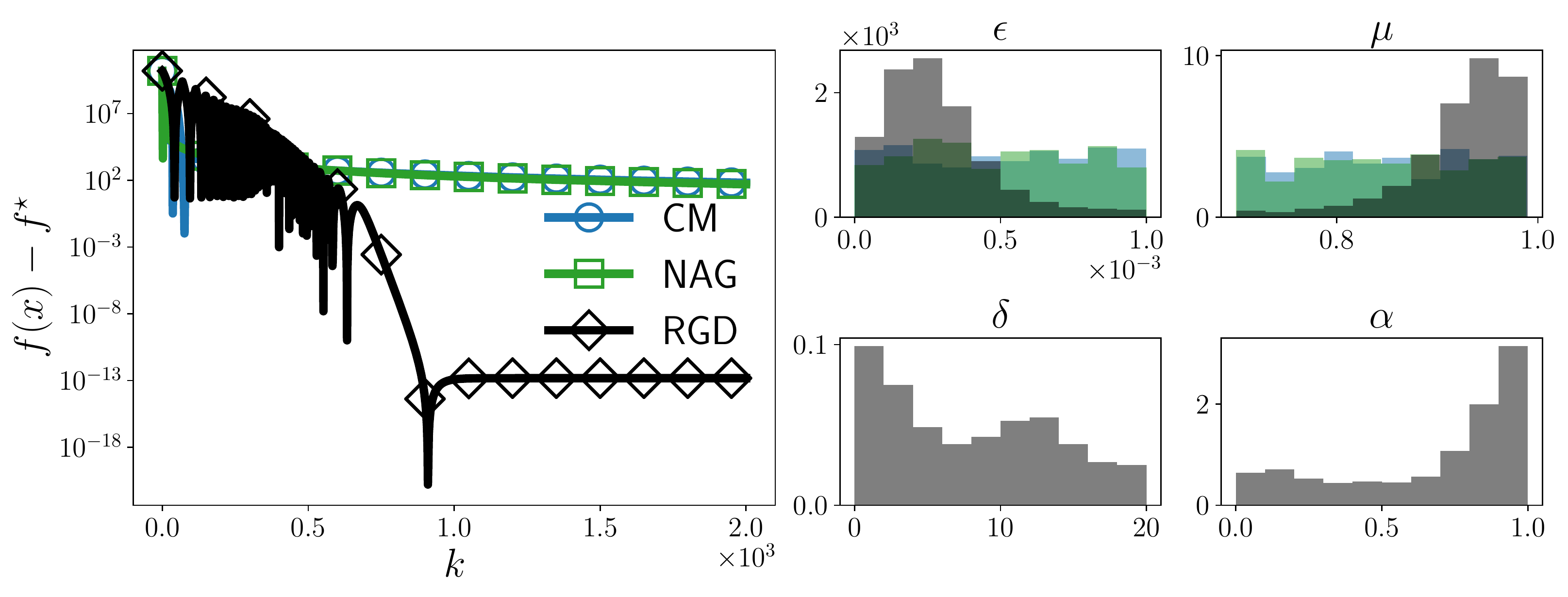}
\caption{\label{fig:chung} \emph{Chung-Reynolds function}: $f(x) \equiv \left(\sum_{i=1}^n x_i^2\right)^2$. The global minimum is at $f(0) = 0$. This function is usually considered on the region $ -100 \le x_i \le 100$. We consider $n=50$ dimensions and initialize at
$x_0=(50,\dotsc,50)$.  Note that RGD was able to improve convergence by controlling the kinetic energy with $\delta > 0$.  We also see the benefits of being conformal symplectic, i.e. $\alpha \to 1$.
}
\end{figure}
\begin{figure}[!h]
\centering
\includegraphics[width=0.33\textwidth,trim={0 0 40 0}]{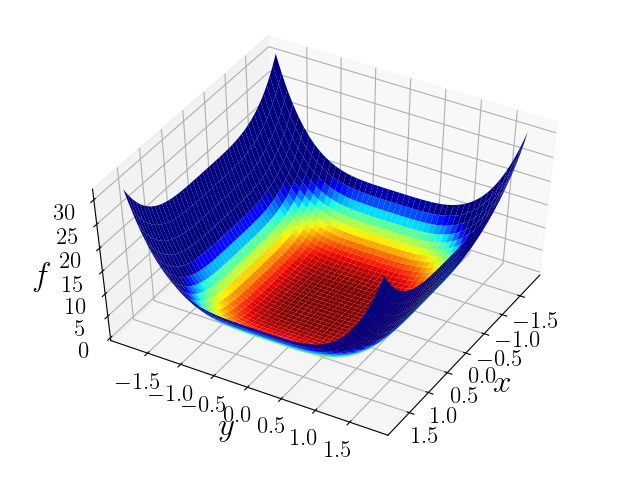}%
\includegraphics[width=0.66\textwidth]{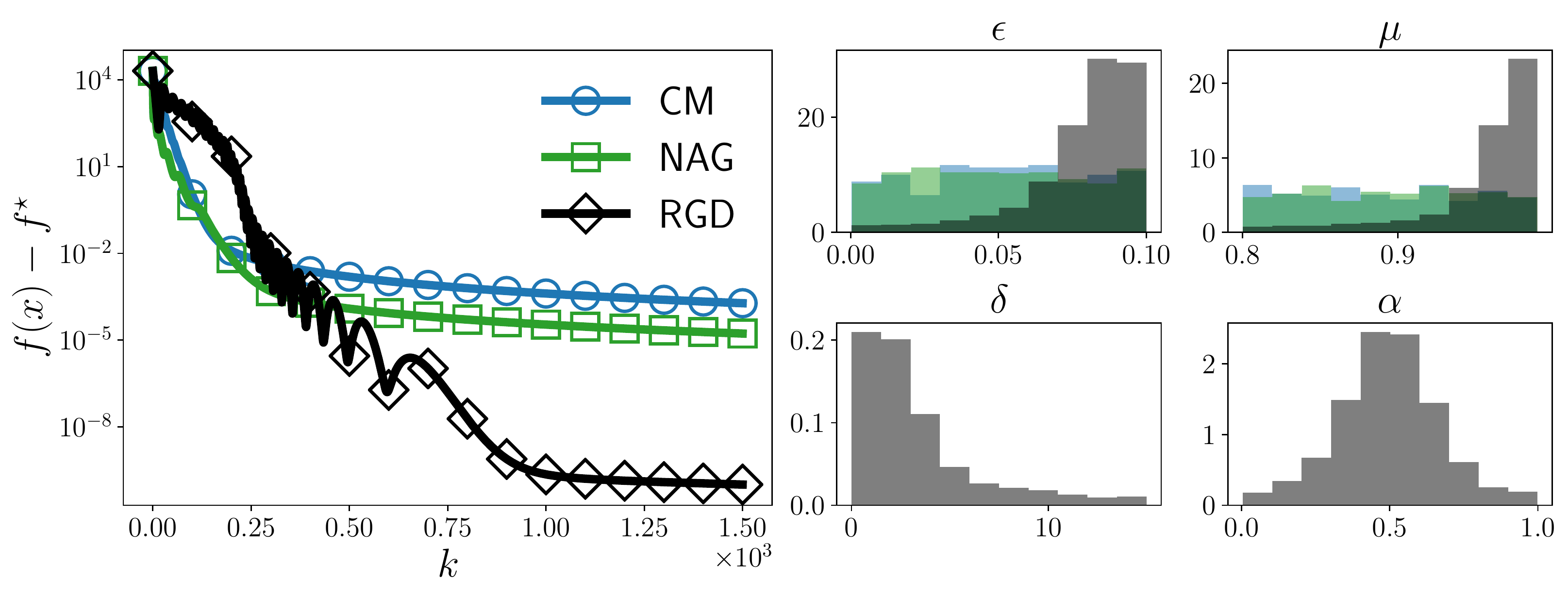}
\caption{\label{fig:quartic} \emph{Quartic function}: $f(x) \equiv \sum_{i=1}^n i x_i^4$. The global minimum is at $f(0) = 0$. This function is usually considered over $-1.28 \le x_i \le 1.28$. We choose $n=50$ dimensions and initialize at $x_0 = (2,\dotsc,2)$.}
\end{figure}
\begin{figure}[!h]
\centering
\includegraphics[width=0.33\textwidth,trim={0 0 40 0}]{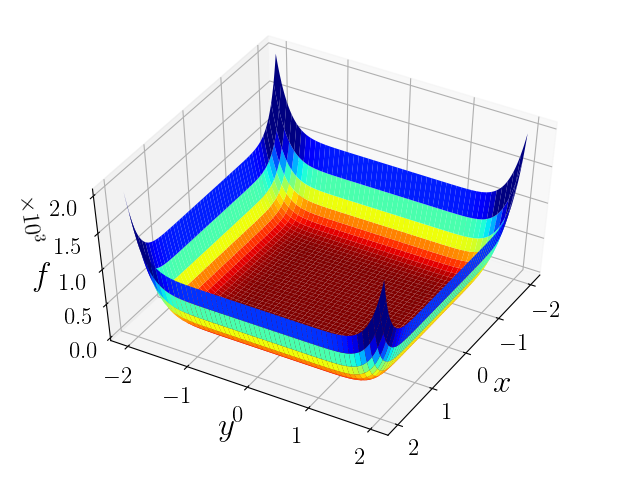}%
\includegraphics[width=0.66\textwidth]{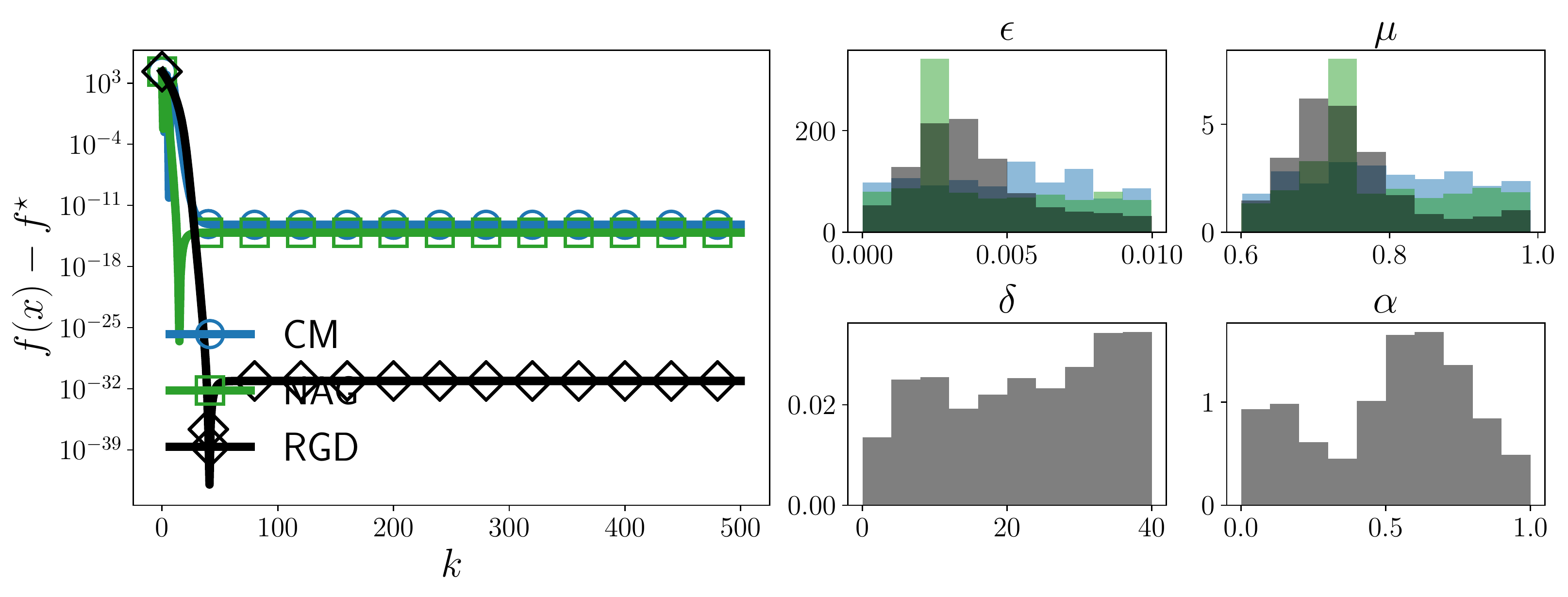}
\caption{\label{fig:schwefel} \emph{Schwefel function}: $f(x) \equiv \sum_{i=1}^n x_i^{10}$. The minimum is at $f(0) = 0$. The function is usually considered over $-10 \le x_i \le 10$. This function grows even stronger than the previous two cases. We consider $n=20$ dimensions and initialize at $x_0 = (2,\dotsc,2)$. Note that $\delta > 0$ is essential to control the kinetic energy and improve convergence.}
\end{figure}
\begin{figure}[!h]
\centering
\includegraphics[width=0.33\textwidth,trim={0 0 40 0}]{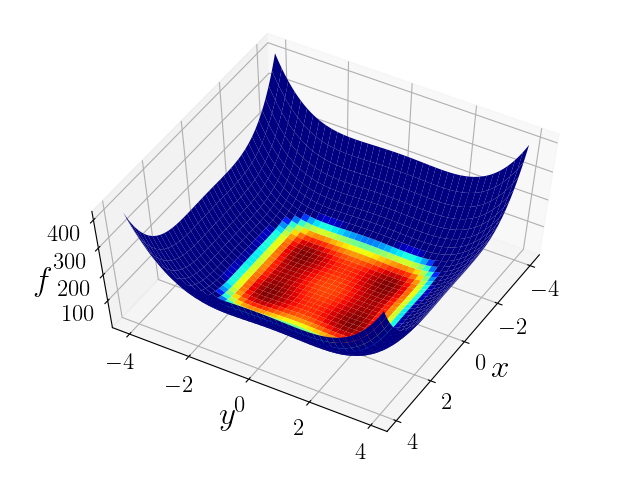}%
\includegraphics[width=0.66\textwidth]{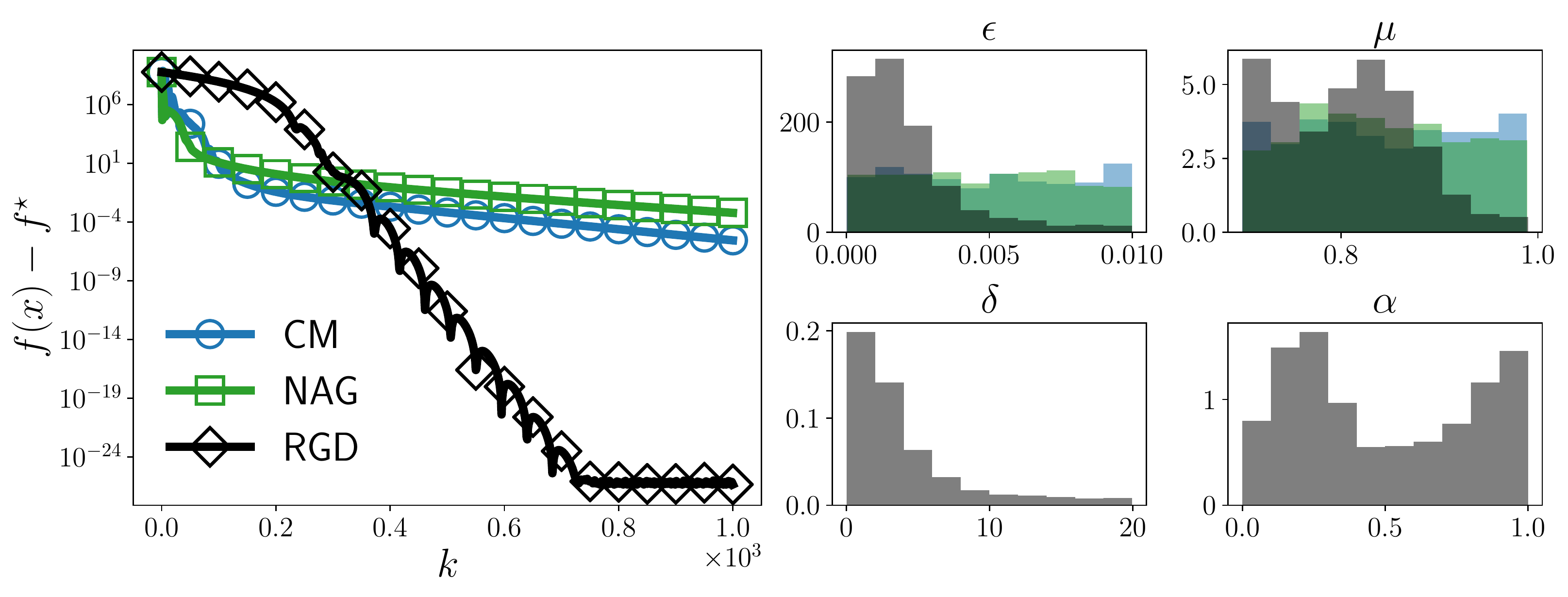}
\caption{\label{fig:qing} \emph{Qing function}: $f(x) \equiv \sum_{i=1}^n (x_i^2 - i)^2$. This function is multimodal, with minimum at $x^\star_i = \pm \sqrt{i}$, $f(x^\star) = 0$. The function is usually studied in the region $-500 \le x_i \le 500$.  We consider $n=100$ dimensions with initialization at $x_0 = (50,\dotsc,50)$.
}
\end{figure}
\begin{figure}[!h]
\centering
\includegraphics[width=0.33\textwidth,trim={0 0 40 0}]{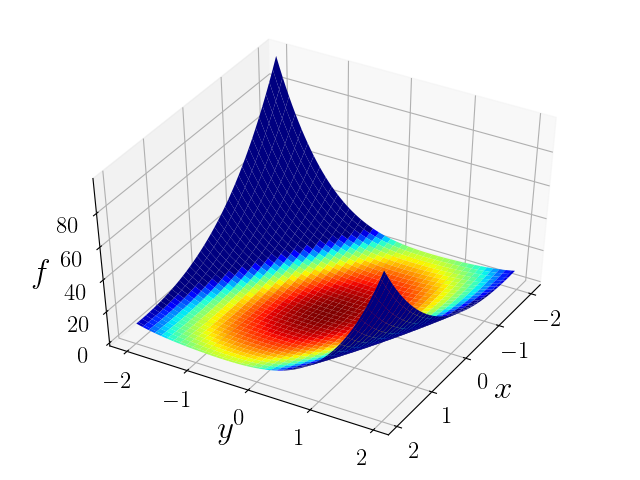}%
\includegraphics[width=0.66\textwidth]{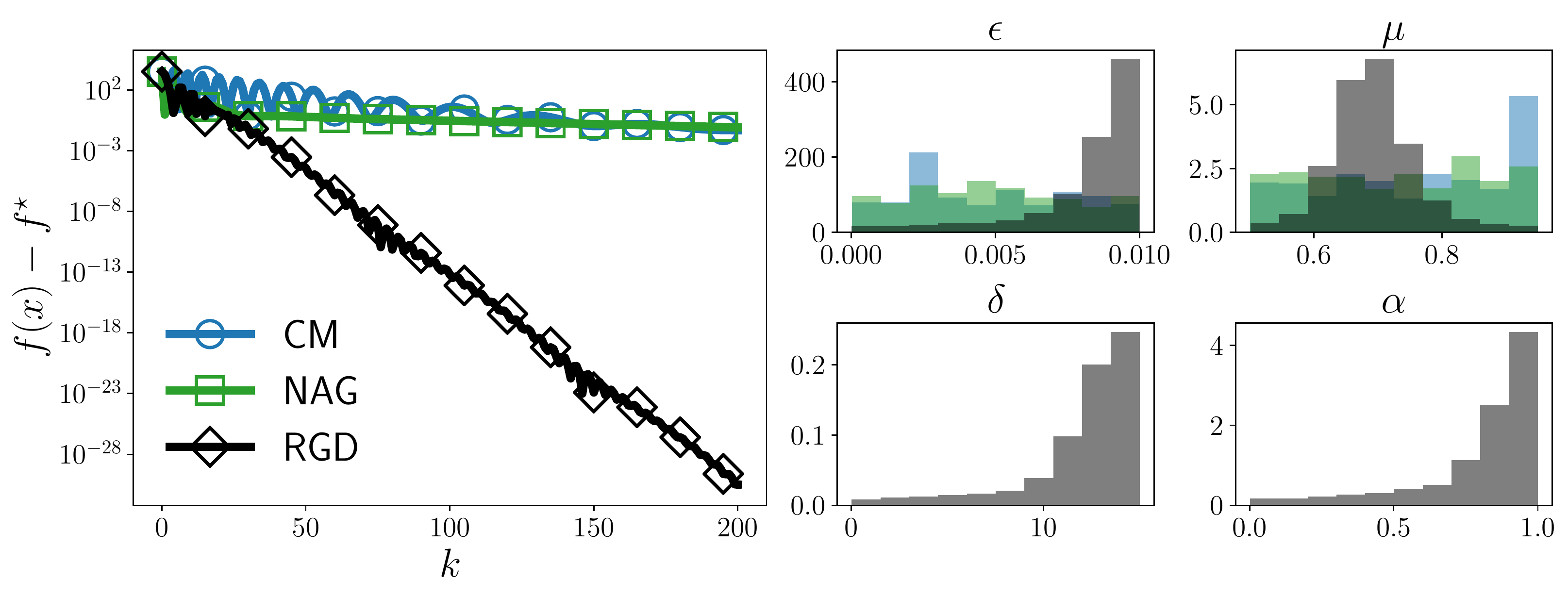}
\caption{\label{fig:zakharov} \emph{Zakharov  function}: $f(x) \equiv \sum_{i=1}^n x_i^2 + \left(\tfrac{1}{2} \sum_{i=1}^n i x_i\right)^2 + \left(\tfrac{1}{2} \sum_{i=1}^n i x_i\right)^4$. The minimum is at $f(0) = 0$. The region of interest is
usually $-5 \le x_i \le 10$. We consider $n=5$ and initialize at $x_0 = (1,\dotsc,1)$. Note that $\delta > 0$ played a dominant role here, and $\alpha \to 1$ as well. RGD successfully minimized this function to high accuracy, contrary to CM and NAG that were unable to get even close to the minimum.}
\end{figure}

\begin{figure}[!h]
\centering
\includegraphics[width=0.33\textwidth,trim={0 0 40 0}]{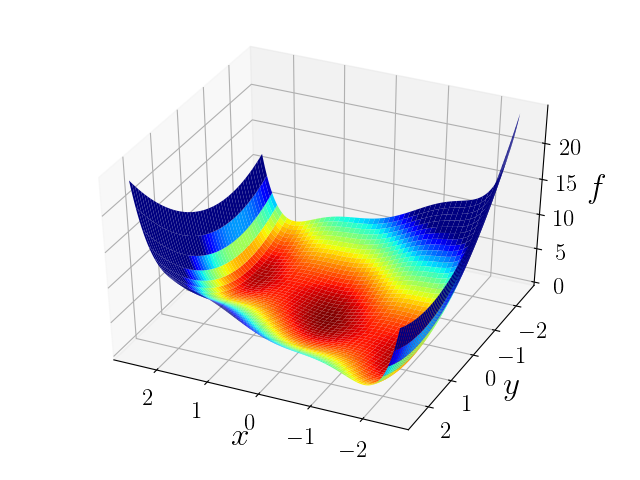}%
\includegraphics[width=0.66\textwidth]{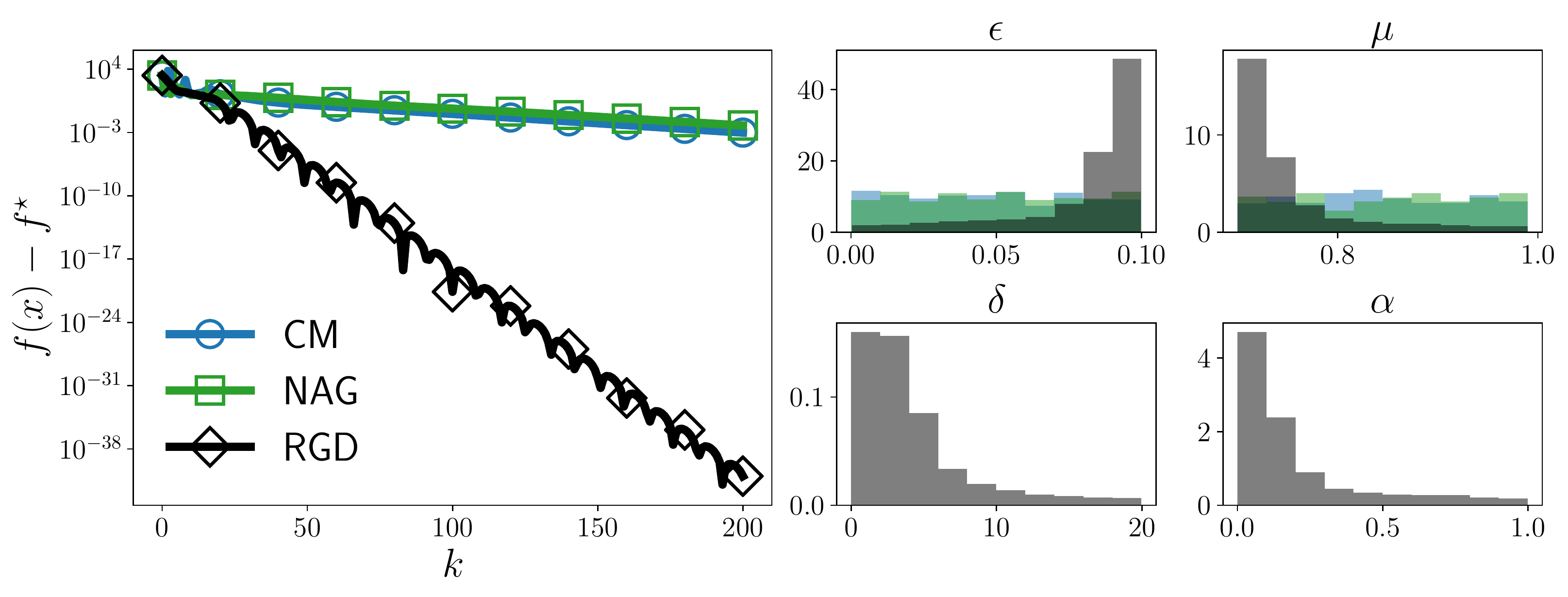}
\caption{\label{fig:camel} Three-hump  camel  back function: $f(x, y) \equiv 2x^2 - 1.05 x^4 + x^6/6 + xy + y^2$. This is a multimodal function with global minimum is at $f(0,0) = 0$. The region of interest is usually $-5 \le x,y \le 5$.
We initialize at $x_0=(5,5)$. The two local minima are somewhat close to the global minimum which makes optimization challenging. Only RGD was able to minimize the function.}
\end{figure}
\FloatBarrier
\begin{figure}[!h]
\centering
\includegraphics[width=0.33\textwidth,trim={0 0 40 0}]{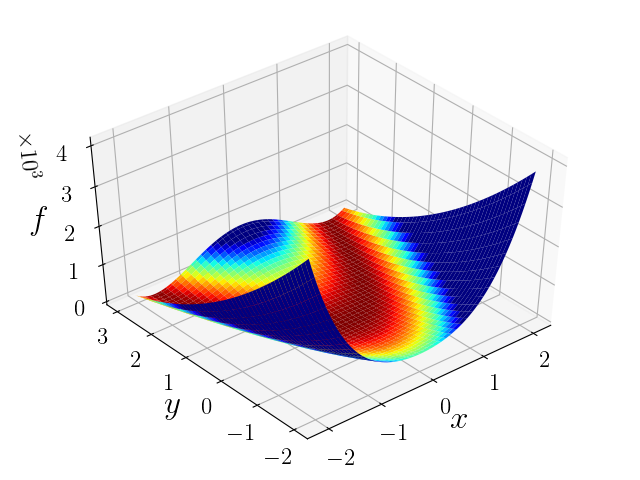}%
\includegraphics[width=0.66\textwidth]{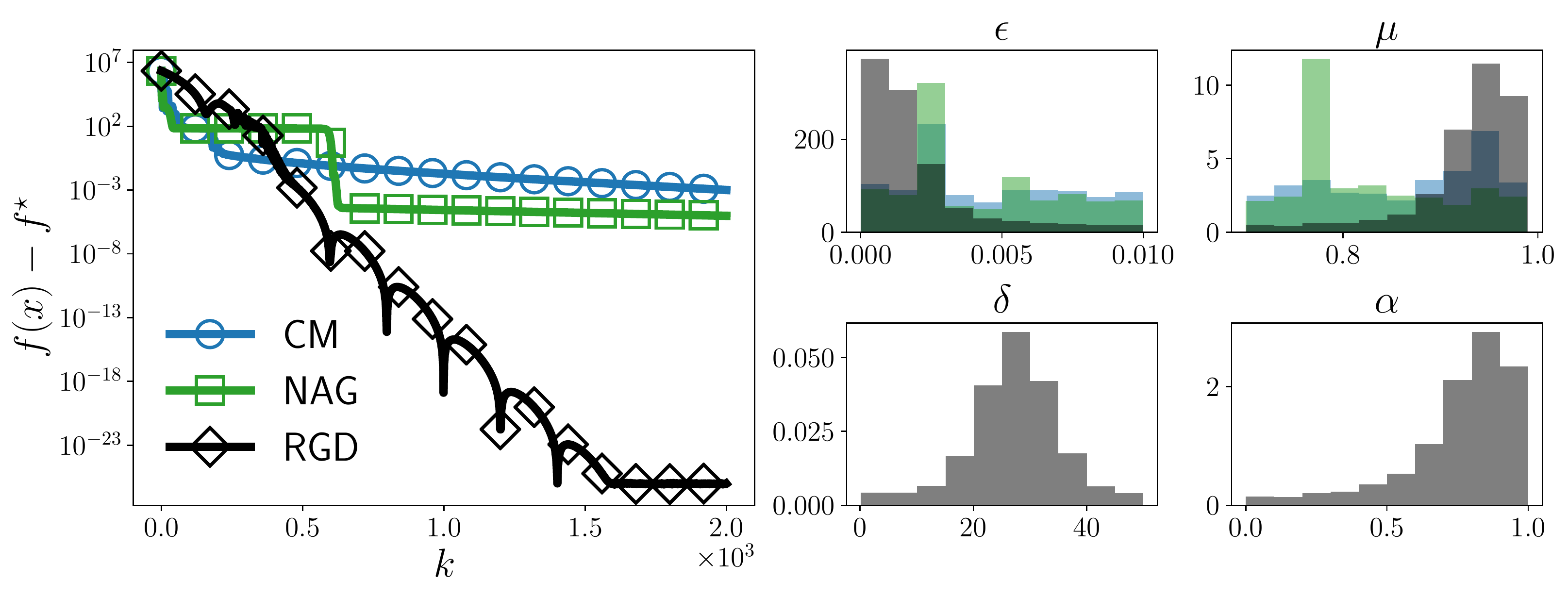}
\caption{\label{fig:rosenbrock} \emph{Rosenbrock  function}: $f(x) \equiv \sum_{i-1}^{n-1} \left( 100(x_{i+1}-x_i^2)^2  + (x_i - 1)^2  \right)$. The global  minimum is at $f(1,\dotsc,1) = 0$. More details about this function was described in Section~\ref{numerical}.  Here we consider $n=1000$ dimensions and initialize at $x_0 = (2.048,\dotsc,2.048)$.
This function is usually studied in the region $-2.048 \le x_i \le 2.048$. Note that $\delta > 0$ was important
for the improved convergence.}
\end{figure}

\bibliography{biblio.bib}

\providecommand{\href}[2]{#2}\begingroup\raggedright\begin{thebibliography}{10}

\bibitem{Polyak:1964}
B.~T. Polyak, ``Some methods of speeding up the convergence of iteration
  methods,'' \href{http://dx.doi.org/10.1016/0041-5553(64)90137-5}{{\em USSR
  Comp. Math. and Math. Physics} {\bfseries 4} no.~5, (1964) 1--17}.

\bibitem{Nesterov:1983}
Y.~Nesterov, ``A method of solving a convex programming problem with
  convergence rate {$O(1/k^2)$},'' {\em Dokl. Akad. Nauk SSSR} {\bfseries 269}
  (1983) 543--547.

\bibitem{Hinton:2013}
I.~Sutskever, J.~Martens, G.~Dahl, and G.~Hinton, ``On the importance of
  initialization and momentum in deep learning,'' in {\em {Int. Conf. Machine
  Learning}}.
\newblock 2013.

\bibitem{Tieleman:2012}
T.~Tieleman and G.~Hinton, ``Lecture 6.5-{RMSprop}: Divide the gradient by a
  running average of its recent magnitude.'' Coursera: Neural Networks for
  Machine Learning, 2012.

\bibitem{Kingma:2015}
D.~P. Kingma and J.~L. Ba, ``Adam: A method for stochastic optimization,'' in
  {\em {Int. Conf. Learning Representations}}.
\newblock 2015.

\bibitem{Duchi:2017}
J.~Duchi, E.~Hazan, and Y.~Singer, ``Adaptive subgradient methods of online
  learning and stochastic optimization,'' {\em {J. Machine Learning Research}}
  {\bfseries 12} (2017) 2121--2159.

\bibitem{Dozat:2016}
T.~Dozat, ``Incorporating {N}esterov momentum into {Adam},'' in {\em Int. Conf.
  Learning Representations, Workshop}.
\newblock 2016.

\bibitem{Candes:2016}
W.~Su, S.~Boyd, and E.~J. Cand{{\`e}}s, ``A differential equation for modeling
  {N}esterov's accelerated gradient method: theory and insights,'' {\em J.
  Machine Learning Research} {\bfseries 17} no.~153, (2016) 1--43.

\bibitem{Wibisono:2016}
A.~Wibisono, A.~C. Wilson, and M.~I. Jordan, ``A variational perspective on
  accelerated methods in optimization,''
  \href{http://dx.doi.org/10.1073/pnas.1614734113}{{\em {Proc. Nat. Acad.
  Sci.}} {\bfseries 113} no.~47, (2016) E7351--E7358}.

\bibitem{Krichene:2015}
W.~Krichene, A.~Bayen, and P.~L. Bartlett, ``Accelerated mirror descent in
  continuous and discrete time,'' in {\em Advances in Neural Information
  Processing Systems}, vol.~28.
\newblock 2015.

\bibitem{Zhang:2018}
J.~Zhang, A.~Mokhtari, S.~Sra, and A.~Jadbabaie, ``Direct {R}unge-{K}utta
  discretization achieves acceleration,'' in {\em {Advances in Neural
  Information Processing Systems}}, vol.~31.
\newblock 2018.

\bibitem{Shi:2018}
B.~Shi, S.~S. Du, M.~I. Jordan, and W.~J. Su, ``Understanding the acceleration
  phenomenon via high-resolution differential equations.'' arXiv:1810.08907
  [math.OC], 2018.

\bibitem{Yang:2018}
L.~F. Yang, R.~Arora, V.~Braverman, and T.~Zhao, ``The physical systems behind
  optimization algorithms,'' in {\em Advances on Neural Information Processing
  Systems}.
\newblock 2018.

\bibitem{Betancourt:2018}
M.~Betancourt, M.~I. Jordan, and A.~C. Wilson, ``On symplectic optimization,''
  \href{http://arxiv.org/abs/1802.03653}{{\ttfamily arXiv:1802.03653
  [stat.CO]}}.

\bibitem{Franca:2018}
G.~Fran{\c c}a, D.~P. Robinson, and R.~Vidal, ``{ADMM} and accelerated {ADMM}
  as continuous dynamical systems,'' {\em {Int. Conf. Machine Learning}} (2018)
  .

\bibitem{Franca:2018b}
G.~Fran{\c c}a, D.~P. Robinson, and R.~Vidal, ``A nonsmooth dynamical systems
  perspective on accelerated extensions of {ADMM},''
  \href{http://arxiv.org/abs/1808.04048}{{\ttfamily arXiv:1808.04048
  [math.OC]}}.

\bibitem{Franca:2019_split}
G.~Fran{\c{c}}a, D.~P. Robinson, and R.~Vidal, ``Gradient flows and accelerated
  proximal splitting methods,''
  \href{http://arxiv.org/abs/1908.00865}{{\ttfamily arXiv:1908.00865
  [math.OC]}}.

\bibitem{Franca:2020}
G.~Fran{\c{c}}a, M.~I. Jordan, and R.~Vidal, ``On dissipative symplectic
  integration with applications to gradient-based optimization,''
  \href{http://arxiv.org/abs/2004.06840}{{\ttfamily arXiv:2004.06840
  [math.OC]}}.

\bibitem{McLachlan:2006}
R.~I. McLachlan and G.~R.~W. Quispel, ``Geometric integrators for {ODEs},''
  \href{http://dx.doi.org/10.1088/0305-4470/39/19/S01}{{\em {J. Phys. A: Math.
  Gen.}} {\bfseries 39} (2006) 5251--5285}.

\bibitem{Forest:2006}
E.~Forest, ``Geometric integration for particle accelerators,''
  \href{http://dx.doi.org/10.1088/0305-4470/39/19/S03}{{\em {J. Phys. A: Math.
  Gen.}} {\bfseries 39} (2006) 5321--–5377}.

\bibitem{Channell:1990}
P.~J. Channell and C.~Scovel, ``Symplectic integration of {H}amiltonian
  systems,'' \href{http://dx.doi.org/10.1088/0951-7715/3/2/001}{{\em
  {Nonlinearity}} {\bfseries 3} (1990) 231--259}.

\bibitem{Quispel:2006}
R.~Quispel and R.~McLachlan, ``Geometric numerical integration of differential
  equations,'' \href{http://dx.doi.org/10.1088/0305-4470/39/19/E01}{{\em {J.
  Phys. A: Math. Gen.}} {\bfseries 39} (2006) }.

\bibitem{McLachlan:2001}
R.~McLachlan and M.~Perlmutter, ``Conformal {H}amiltonian systems,''
  \href{http://dx.doi.org/10.1016/S0393-0440(01)00020-1}{{\em {Journal of
  Geometry and Physics}} {\bfseries 39} (2001) 276--300}.

\bibitem{Moore:2016}
A.~Bhatt, D.~Floyd, and B.~E. Moore, ``Second order conformal symplectic
  schemes for damped {H}amiltonian systems,''
  \href{http://dx.doi.org/10.1007/s10915-015-0062-z}{{\em {Journal of
  Scientific Computing}} {\bfseries 66} (2016) 1234--1259}.

\bibitem{Muehlebach:2020}
M.~Muehlebach and M.~I. Jordan, ``Optimization with momentum: dynamical,
  control-theoretic, and symplectic perspectives,''
  \href{http://arxiv.org/abs/2002.12493}{{\ttfamily arXiv:2002.12493
  [math.OC]}}.

\bibitem{Lu:2017}
X.~Lu, V.~Perrone, L.~Hasenclever, Y.~W. Teh, and S.~J. Vollmer, ``Relativistic
  {M}onte {C}arlo,'' {\em 20th Int. Conf. Articial Intelligence and Statistics}
  (2017) .

\bibitem{Livingstone:2017}
S.~Livingstone, M.~F. Faulkner, and G.~O. Roberts, ``Kinetic energy choice in
  {H}amiltonian/hybrid {M}onte {C}arlo,''
  \href{http://arxiv.org/abs/1706.02649}{{\ttfamily arXiv:1706.02649
  [stat.CO]}}.

\bibitem{Maddison:2018}
C.~J. Maddison, D.~Paulin, Y.~W. Teh, B.~O'Donoghue, and A.~Doucet,
  ``Hamiltonian descent methods,''
  \href{http://arxiv.org/abs/1809.05042}{{\ttfamily arXiv:1809.05042
  [math.OC]}}.

\bibitem{Flanders}
H.~Flanders, {\em Differential Forms with Applications to the Physical
  Sciences}.
\newblock Dover, 1989.

\bibitem{Hairer}
E.~Hairer, C.~Lubich, and G.~Wanner, {\em Geometric Numerical Integration}.
\newblock Springer, 2006.

\bibitem{Dessler:1988}
U.~Dessler, ``Symmetry property of the {L}yapunov spectra of a class of
  dissipative dynamical systems with viscous damping,''
  \href{http://dx.doi.org/10.1103/PhysRevA.38.2103}{{\em {Phys. Rev. A}}
  {\bfseries 38} (1988) 2103}.

\bibitem{Maro:2017}
S.~Mar{\`{o}} and A.~Sorrentino, ``{A}ubry-{M}ather theory for conformally
  symplectic systems,'' \href{http://dx.doi.org/10.1007/s00220-017-2900-3}{{\em
  {Commun. Math. Phys.}} {\bfseries 354} (2017) 775--808}.

\bibitem{Hyperopt}
J.~Bergstra, D.~Yamins, and D.~D. Cox, ``Making a science of model search:
  Hyperparameter optimization in hundreds of dimensions for vision
  architectures,'' {\em Int. Conf. Machine Learning} (2013) .

\bibitem{code}
G.~Fran{\c{c}}a, ``Relativistic gradient descent ({RGD}),'' 2020.
\newblock \url{https://github.com/guisf/rgd.git}.

\bibitem{Rosenbrock:1960}
H.~H. Rosenbrock, ``An automatic method for finding the greatest or least value
  of a function,'' \href{http://dx.doi.org/10.1093/comjnl/3.3.175}{{\em {The
  Computer Journal}} {\bfseries 3} no.~3, (1960) 175–--184}.

\bibitem{Goldberg:1989}
D.~E. Goldberg, {\em Genetic Algorithms in Search, Optimization and Machine
  Learning}.
\newblock {Addison-Wesley}, 1989.

\bibitem{Kok:2009}
S.~Kok and C.~Sandrock, ``Locating and characterizing the stationary points of
  the extended {R}osenbrock function,''
  \href{http://dx.doi.org/10.1162/evco.2009.17.3.437}{{\em {Evolutionary
  Computation }} {\bfseries 17} no.~3, (2009) 437--453}.

\bibitem{Montanari:2009}
R.~H. Keshavan, A.~Montanari, and S.~Oh, ``Matrix completion from a few
  entries,'' \href{http://dx.doi.org/10.1109/TIT.2010.2046205}{{\em IEEE Trans.
  Information Theory} {\bfseries 56} (2009) }.

\bibitem{Jamil:2013}
M.~Jamil and X.-S. Yang, ``A literature survey of benchmark functions for
  global optimization problems,''
  \href{http://dx.doi.org/10.1504/IJMMNO.2013.055204}{{\em Int. J. Math.
  Modelling and Num. Optimisation} {\bfseries 4} no.~2, (2013) 150--194}.

\end{thebibliography}\endgroup

\end{document}